\documentclass[12pt]{article}
\usepackage[utf8]{inputenc}
\usepackage{times}
\usepackage{array}
\usepackage[english]{babel}
\usepackage{bm,mathrsfs,amsmath,amssymb,amsthm,color,comment}
\usepackage{amsfonts,commath,mathtools,enumerate,enumitem,esint}
\input epsf

\usepackage{hyperref}
\hypersetup{urlcolor=blue, citecolor=blue, linkcolor=blue}

\usepackage{cite}

\usepackage[textsize=small]{todonotes}
\definecolor{wineRed}{rgb}{0.7,0,0.3}
\definecolor{grandBleu}{rgb}{0,0,0.8}
\definecolor{darkGreen}{rgb}{0,0.5,0}
\definecolor{blueViolet}{rgb}{0.4,0,1.0}
\definecolor{bloodOrange}{rgb}{0.85,0.05,0}
\definecolor{mycolor}{rgb}{0.8,0,0.2}
\definecolor{}{rgb}{0.8,0,0.2}
\def\ryota #1{#1} 
\newcommand{\KSh}[1]{{\color{grandBleu}#1}}

\theoremstyle{definition}

\newtheorem{lemma}{Lemma}[section]

\newtheorem{cor}{Corollary}[section]

\newtheorem{thm}{Theorem}[section]

\newtheorem{remark}{Remark}[section]

 \newtheorem{con}[thm]{Conjecture}
 \newtheorem{lem}[thm]{Lemma}
 \newtheorem{prop}[thm]{Proposition}
 \theoremstyle{definition}
 \newtheorem{defn}[thm]{Definition}

  \def\bne{{\bf e}}
\def\di{\hbox{div}\,}

\def\Lip{\hbox{Lip}\,}
\def\bR{\mathbb{R}}

\def\d{\partial}

\def\cE{{\mathcal{E}}}
\def\cH{{\mathcal{H}}}

\def\dps{\displaystyle}
\def\txs{\textstyle}
\def\Sgn{\hbox{Sgn}}
\def\ds{\displaystyle}

\numberwithin{equation}{section}

\begin{document}
%opening
\title{On boundary detachment phenomena for the total variation flow
  with %the
  dynamic boundary conditions}
 
 \medskip
\author{Yoshikazu Giga${}^1$, Ryota Nakayashiki${}^2$, Piotr Rybka${}^3$, Ken Shirakawa${}^4$\\
${}^1$ Graduate School of Mathematical Sciences\\
University of Tokyo\\
Komaba 3-8-1,  Tokyo 153-8914,
Japan\\
e-mail: {\tt labgiga@ms.u-tokyo.ac.jp}
\medskip\\
${}^2$ Department of Mathematics and Informatics,\\ Graduate School of Science, Chiba University\\ Yayoi-cho 1-33, Inage-ku, Chiba, 263-8522, Japan\\
e-mail: {\tt nakayashiki1108@chiba-u.jp}
\medskip\\
${}^3$ Institute of Applied Mathematics and Mechanics,
Warsaw University\\ ul. Banacha 2, 07-097 Warsaw, Poland\\
fax: +(48 22) 554 4300, e-mail: {\tt rybka@mimuw.edu.pl}
\medskip\\
${}^4$ Department of Mathematics, Faculty of Education, Chiba University\\ Yayoi-cho 1-33, Inage-ku, Chiba, 263-8522, Japan\\
e-mail: {\tt sirakawa@faculty.chiba-u.jp}
}

\maketitle

\begin{abstract}
We combine the total variation flow suitable for crystal modeling and image analysis with the dynamic boundary conditions. We analyze the behavior of facets at the parts of the boundary where these conditions are imposed. We devote particular attention to the radially symmetric data. We observe that the boundary layer detachment actually can happen at concave parts of the boundary.
\end{abstract}



\section{Introduction}
We consider the total variation flow with the dynamic boundary condition, possibly mixed with the Neumann boundary condition, which can be formally written as follows,
\begin{equation}\label{eq1}
 \begin{array}{ll}
     u_t = \di \left({\ds\frac{\nabla u}{|\nabla u|}}\right) & \hbox{for }(x,t)\in \Omega\times (0,T) =:Q_T;\\[2.5ex]
     {\tau} v_t = - \frac{\nabla u}{|\nabla u|}  \cdot {\nu}& \hbox{for }(y,t)\in \Gamma\times (0,T) =:S_T;\\[1.5ex]
  \frac{\d u}{\d {\nu}} = 0& \hbox{for }(y,t)\in (\d\Omega\setminus\Gamma)\times (0,T);\\[1.5ex]
  u(x,0) = u_0(x) & \hbox{for }x \in \Omega;\\[1.5ex]
  v(y,0) = v_0(y)& \hbox{for }y\in \Gamma.
 \end{array}
\end{equation}
Here 
%$\Omega\subset \bR^N$, is an open set with \KSh{a boundary}sufficiently smooth boundary. 
%We will restrict our attention to $N=1$ and $N=2$. 
$ \Omega \subset \mathbb{R}^N $ is a bounded spatial domain of dimension $ N \in \mathbb{N} $, and when $ N > 1 $, the boundary $ \d\Omega $ is supposed to be sufficiently smooth.
Moreover, $\Gamma\subset \d\Omega$, possibly $\Gamma =\d\Omega$ is a part of the boundary with positive $\cH^{N-1}$-measure. The outer normal to $\d\Omega$ at $x$ is denoted by
$\nu(x)$. %In principle, $W(\cdot)$ could be a convex one-homogeneous function, but for the sake of simplicity we take $W$ to be a norm on $\bR^N$, the discussion presented below is for
%most of the time this will be 
%the Euclidean norm. 

It is  well-known that the total variation flow leads to the creation of facets, i.e. persistent flat parts of solutions. Here, we study the interactions of facets  at their junction with the boundary at $\Gamma$, where  the dynamic boundary conditions are specified. 
%We also include the dependence on parameter $\tau$, which will make a link to the Neumann or Dirichlet boundary data, at least formally, where $\tau=0$ or $\tau =+\infty$.

Even though the total variation flow with the Dirichlet boundary conditions was studied by a number of authors, see \cite{mazon,MR2139257, %Moll's paper, 
bogelein,Matusik}, the details of the boundary behavior were not extensively discussed. In particular this applies to the evolution %. Some of them were not interested in the 
of facets touching the boundary. It is worth emphasizing that the authors of \cite{mazon,bogelein} invested a lot of effort into finding the correct notion of the solution. This is particularly true for \cite{bogelein}, where quite general time-dependent Dirichlet boundary data are considered. 

It is worth noticing that it is known, see \cite{andreu, MR2139257, Matusik},
that in general, the  Dirichlet boundary data
may not be attained in the sense of trace.

We are interested in a phenomenon, which was studied in \cite{oliker} for the Dirichlet problem of graphs evolving by the mean curvature. The authors showed there that a boundary layer may detach from the solution in the bulk. This phenomenon is attributed to the  lack of uniform parabolicity of the mean curvature flow for graphs. Obviously, this lack of uniform parabolicity occurs here too.

It is worth mentioning that the problem of the loss of the boundary conditions was studied by a number of authors in the context of viscosity solutions to parabolic equations with nonlinearities involving gradient of solution.
A good example of such research is \cite{quaas}, where also a historical account is presented. However, the nature of the phenomenon studied in \cite{quaas} is different from what we study here. In case of first order Hamilton-Jacobi
equations we refer the reader to \cite{Elliott-Giga-Goto} for earlier study on unattainability of the
boundary condition.

The non-attainment of boundary condition is a common problem for the steady states of the total variation flow. They are better known as solutions to the least gradient problem. Special geometric restrictions must be imposed on the domain $\Omega$ as well as on the boundary datum $f$ to ensure attainment, see \cite{sternberg, segura, grs}.

Here, 
we are observing  a similar phenomenon of the boundary layer detachment. We introduce a family of evolution problems with dynamic boundary condition indexed by parameter $\tau\in (0,\infty)$. %The merit is to f  wit
By formally taking the limit as $\tau\to 0$ we recover the Neumann data, while the limit $\tau\to \infty$ yields the Dirichlet boundary conditions. However, a rigorous statement is outside the scope of this paper.

The study of the dynamic boundary conditions has a long history. In the early days, solvability of uniformly parabolic equations with dynamic boundary conditions was discussed by Escher, \cite{escher}.  
For fully nonlinear (possibly degenerate) parabolic equations,
Barles established a quite general comparison result in \cite[Sect. II]{barles1} and \cite[Sect. 3]{barles2} for a general nonlinear dynamic boundary condition. The mean curvature flow for a level set,
under a dynamic boundary condition is discussed in Giga-Hamamuki \cite{giga-hamamuki}, which is
not included in papers of Barles, \cite{barles1,barles2}.

For further development in case of uniformly parabolic
problems we refer to \cite{denk-pruss, vazquez, MR2434980}. More recently this topic was studied in \cite{MR3434333, GMS, MR3342513, MR3661429}. It is worth emphasizing that the dynamic boundary conditions were studied in relation to Stefan problem, cf. \cite{aiki}, Allen--Cahn type equations \cite{sprekels-wu,MR2434980,MR3661429,MR3434333} or Cahn--Hilliard equations, \cite{MR2397309,MR3342513,GMS}. 

Finally, we comment on the physical background of our system \eqref{eq1}. The dynamic boundary conditions kindred to this study are found in the previous works of Stefan problems, e.g. \cite{aiki,savare}, and in particular, our dynamic boundary condition can be characterized as a singular limit of transmitted parabolic problems studied in \cite{savare}. Meanwhile, the singular diffusion as in \eqref{eq1} is associated with a phase transition model of mesoscale, which was proposed by Visintin \cite[Chapter 6, page 176]{visintin}. In view of these, our system \eqref{eq1} can be regarded as a basic problem for a mesoscale phase transition model, that takes into account interactive phase-exchanges reproduced by the dynamic boundary condition.

%Finally, we comment on the physical explanation of the dynamic boundary condition. A possible one is on the ground of the phase transition theory and may be found in Visintin's book, \cite{visintin}, in Chapter 6, page 176. It \textcolor{blue}{was} derived as a singular limit of the transmission problem by Savar\'e-Visintin, see \cite{savare}. Further development can be found in \cite{aiki}.

Our goal in this paper is to study instances of occurrence of the ``boundary layer detachment
phenomenon'' %or boundary layer occurs for facets for 
in the case of the total variation flow
under the dynamic boundary condition on a part of the boundary called $\Gamma$. 
More precisely, we investigate the evolution of the persistent facets touching $\Gamma$, such facets are called {\it calibrable}.

Moreover, if the solution is continuous at points of $\Gamma$, i.e. the facet moves with the same velocity as the boundary value, then we call such a calibrable facet {\it coherent}.
%\todo{It would be better to replace this part by more readable phrase.}
%The former phenomenon is called calibrable while the latter phenomenon coherent. 

If a facet does not touch the
boundary, its calibrability is well studied, especially when the facet is
convex as well as the solution. %is convex. 
In fact, the calibrability of a facet $F$ is
equivalent to saying that $F$ is a Cheeger set, i.e. $F$  minimizes the Cheeger quotient,  
$\lambda=|\partial F|/|F|$, among all subsets. Moreover, it is the same as %equivalent to
saying that the inward mean curvature of $F$, $\kappa$, is dominated by the Cheeger
quotient $\lambda$, %=|\partial F|/|F|$, 
see \cite{alter,BeCaNo}.
%\todo{\KSh{ I leave the decision on (1) of Prof. Giga's comment up to Piotr. KS}}

Under some technical conditions we show that %the 
facet  $F$ is
calibrable and coherent if the sum of inward principal curvature of
$\partial\Omega$ near the intersection of $F$ and $\Gamma$ is greater than $-1/\tau$.
In one dimensional case $N=1$, we show all facets are calibrable and coherent.
However, in $N=2$, if one considers annuli, the facet touching the inner circle
may not be coherent and boundary detachment phenomenon actually occurs. In order to
derive these results, we first clarify the definition of a solution by taking a
correct energy and show the well-posedness of the problem. We next calculate
canonical restriction of subdifferential of the energy. Although the general
strategy is similar to those in \cite{mazon,andreu}, it is nontrivial to implement the
strategy.

Let us describe the content of this paper. 
We present here a general existence result for (\ref{eq1}). For this purpose we use the nonlinear semigroup theory developed  by K\=omura \cite{Komura} and Br\'ezis \cite{MR0348562}. The main step in this direction is the identification of (\ref{eq1}) as a gradient flow of an energy functional $E$. It turns out that the natural definition of $E: L^2(\Omega) \times L^2(\Gamma) \to \bR\cup\{+\infty\}$ is as follows,
\begin{equation*}%\label{2dfE}
 E(u,v) = \left\{
\begin{array}{ll}
    {\ds\int_\Omega |Du| + \int_\Gamma|\gamma u -v|\,d\cH^{N-1}}& \hbox{if } (u,v)\in BV(\Omega)\times L^2(\Gamma),\\[2ex]
% +\infty & \hbox{for } (u,v)\in (L^2(\Omega) \setminus BV(\Omega))\times L^2(\Gamma).
    +\infty & \hbox{otherwise. }
\end{array}
 \right. 
\end{equation*}
In Section \ref{synergy}, we study the lower semi-continuity of $E$ and related problems, because this is the precondition of the  nonlinear semigroup theory. In Section \ref{sevolt}, we state and prove the existence of solutions to (\ref{eq1}). Here, our point of departure is the observation that (\ref{eq1}) is a gradient flow of $E$. In fact, if $\tau =1$, then (\ref{eq1}) is the  gradient flow of $E$ with respect to the standard inner product in $H= L^2(\Omega) \times L^2(\Gamma)$. We notice that (\ref{eq1}) may be viewed as a  
gradient flow of $E$ with respect to a non-standard inner product in $H$, given by formula $((u_1,v_1),(u_2,v_2))_\tau = \int_\Omega u_1 u_2\,dx + \tau \int_\Gamma v_1 v_2 \,d\cH^{N-1}$. We comment on this in Section \ref{sevolt}. 

We also take advantage of the structure of $E$ to notice the order preserving property of the flow and the comparison principle. This is also done in Section \ref{sevolt} and the analysis is based on the work by 
Br\'ezis \cite{brezis-thesis} and Kenmochi \cite{Kenmochi}.

A very important part of the analysis, which on the one hand is technical, on the other hand is necessary for the study of facet evolution is the  identification of the subdifferential, $\partial E$, and its canonical selection. This is performed in Section \ref{ssubd}. This section  closes with the remark on the relationship between the subdifferentials with respect to $(\cdot, \cdot)_\tau$ for different values of $\tau$. 

Section \ref{42CC} prepares the tools for the analysis of facets. In particular, we adjust the notion of calibrability to the present setting, when we pay particular attention to the behavior of facets, touching the boundary of $\d\Omega$ along $\Gamma$, where the dynamic boundary condition is set.

We also introduce the notion of  coherency, which is useful, when we wish to address the phenomenon of the boundary layer detachment. We also state there sufficient and necessary condition for calibrability or coherency.

We study a number of explicit examples, which show different types of behavior. Section \ref{sone} offers an analysis of a one dimensional problem as a warm-up. In this case no boundary layer detachment occurs. The radially symmetric two-dimensional problems are treated in Section \ref{stud}. We notice that a general Theorem \ref{MC} and its Corollary \ref{cor6.1} imply that if $\Gamma=\d\Omega$, where $\Omega$ is a ball, then radially symmetric facets touching $\Gamma$ will be coherent, i.e. no boundary detachment occurs. The situation is different, when we consider an annulus with inner radius $r_0$ and $\Gamma = \d B(0,r_0)$. In this case we pinpoint the situation of the boundary layer detachment.

\section{Preliminaries}

In this section, we begin with the basic notation used throughout this paper.

For an abstract Banach space $X$, we denote by $\|\cdot\|_{X}$ the norm of $X$, and when $X$ is a Hilbert space, we denote by $(\,\cdot\,,\cdot\,)_{X}$ the inner product of $X$. In particular, in cases of Euclidean spaces, we uniformly denote by $ |\cdot| $ the Euclidean norm, and we use ``\,$ \cdot $\,'' to denote the standard scalar product of two vectors.  
Additionally, for fixed dimensions $ d, \ell \in \mathbb{N} $ and a bounded open set $ U \subset \mathbb{R}^d $, we %uniformly 
%\todo{\PR{I deleted 'uniformly' feeling that it is redundant.}} 
denote by $ \|\cdot\|_\infty $ the supremum-norm in $ L^\infty(U, \mathbb{R}^d) $, i.e. $ \| w \|_\infty := \mathrm{ess }\sup_{x \in U} |w(x)| $, for $ w \in L^\infty(U, \mathbb{R}^d) $.

%\medskip

For any proper functional $ \Psi : X \rightarrow(-\infty, \infty] $ on
a Hilbert
%Banach
space $ X $, we denote by $ D(\Psi) $ the effective domain of $ \Psi $, i.e. $ D(\Psi) := \left\{ \begin{array}{l|l}
w \in X & \Psi(w) < \infty
\end{array} \right\} $.
%\medskip

For any proper lower semi-continuous (l.s.c., in short) and convex function $\Phi$ defined on a Hilbert space $X$, we denote the subdifferential of $\Phi$ by $\partial \Phi$. The subdifferential $\partial \Phi$ corresponds to a weak differential of $\Phi$, and in fact it is  a maximal monotone graph in the product space $ X \times X $. More precisely, for each $w_{0} \in {X}$, the value $\partial \Phi(w_{0})$ of the subdifferential at $w_{0}$ is defined as a set of all elements $\eta_0 \in {X}$ which satisfy the following variational inequality:
\begin{equation*}
(\eta_0, w-w_{0})_{{X}} \le \Phi(w) - \Phi(w_{0})\ \ \mbox{for any}\ w \in D(\Phi).
\end{equation*}
The set $D(\partial \Phi) := \left\{ \begin{array}{l|l} 
w \in {X} & \partial\Phi(w) \neq \emptyset\end{array} \right\}$ is called the domain of $\partial\Phi$. We often use the notation ``$ (w_{0},\eta_0) \in \partial\Phi $ in $ {X} \times {X} $\,'', to mean that ``$ \eta_0 \in \partial\Phi(w_{0})$ in ${X}$ with $w_{0} \in D(\partial\Phi)$'', by identifying the operator $\partial\Phi$ with its graph in ${X} \times {X}$.

\begin{remark}\label{Rem.conv}
An example of a subdifferential is the following set-valued sign function $ \Sgn^d :\mathbb{R}^d \to 2^{\mathbb{R}^d} $, given as:
\begin{equation*}
\omega \in \mathbb{R}^d \mapsto \Sgn^d(\omega):= \left\{ \begin{array}{l}
\displaystyle\frac{\omega}{|\omega|}, \mbox{ if $ \omega \ne 0 $,}
\\[3ex]
\bigl\{ \tilde{\omega} \in \mathbb{R}^d \,\bigl|\, |\tilde{\omega}| \leq 1 \bigr\}, \mbox{ if $ \omega = 0 $.}
\end{array}
\right.
\end{equation*}
It is easy to check that  the set-valued function $ \Sgn^d $ coincides with the subdifferential of the Euclidean norm $ |{}\cdot{}|: \omega \in \mathbb{R}^d \mapsto |\omega|=\sqrt{\omega \cdot \omega} \in [0,\infty) $. 
\end{remark}

%\paragraph{Notations in basic measure theory. (cf. \cite{MR1857292, MR2192832})}

%For any $ d \in \mathbb{N} $, let $ \mathcal{L}^d $ be the $ d$-dimensional Lebesgue measure. Also, unless specified otherwise, the measure theoretical phrases, such as ``a.e.'', ``$dt$'', ``$dx$'', and so on, are with respect to the Lebesgue measure in the corresponding dimension. Meanwhile, in the observations on a $ C^1$-surface $ S $, the phrase ``a.e.'' is with respect to the Hausdorff measure that is $\sigma$ finite on $S$. 

%each corresponding Hausdorff dimension. 

%Let $ d, \ell \in \mathbb{N} $ be constants of dimensions, let $ \mu $ be a positive measure on $ \mathbb{R}^d $, and let $ \lambda $ be a $ \mathbb{R}^\ell $-valued measure on $ \mathbb{R}^d $. Then, we denote by $ |\lambda| $ the total variation measure of $ \lambda $, and denote by $ \lambda_\mu^a $, $ \lambda_\mu^s $ and $ \frac{\lambda}{\mu} $, respectively, the absolutely continuous part, the singular part and the Radon-Nikod\'{y}m density of $ \lambda $ for $ \mu $, i.e.:
%\begin{equation*}
%\lambda = \lambda_\mu^a +\lambda_\mu^s \mbox{ and } \lambda_\mu^a = {\textstyle \frac{\lambda}{\mu}} \, \mu, \mbox{ in the sense of measure.}
%\end{equation*}

%For a (Lebesgue) measurable function $ u : B \rightarrow[-\infty, \infty] $ on a Borel set $ B \subset \mathbb{R}^d $, we denote by $ [u]^+ $ and $ [u]^- $ the positive part $ \mathcal{T}_0^\infty \circ u : B \rightarrow[0, \infty] $ and the negative part $ -\mathcal{T}_{-\infty}^0 \circ u : B \rightarrow[0, \infty] $, respectively.

\paragraph{Notations in $BV$-theory. (cf. \cite{MR1857292, MR2192832, MR3409135, MR0775682})}
For any $ d \in \mathbb{N} $, we denote the $ d$-dimensional Lebesgue measure by $ \mathcal{L}^d $. The measure theoretical phrases, such as ``a.e.'', ``$dt$'', ``$dx$'', etc are with respect to the Lebesgue measure in the corresponding dimension, unless specified otherwise.  

%\todo{\KSh{It would be better to revert this part for rigorous mathematical discussions, as in Proposition \ref{s5-pr-cannon}}}

Let $ d\in \mathbb{N} $ be a fixed dimension and let $ U \subset \mathbb{R}^d $ be a bounded open set. We denote by $ \mathcal{M}(U) $ (resp. $ \mathcal{M}_{\rm loc}(U) $) the space of all finite Radon measures (resp. the space of all Radon measures) on $ U $. In general, the space $ \mathcal{M}(U) $ (resp. $ \mathcal{M}_{\rm loc}(U) $) is known as the dual of the Banach space $ C_0(U) $ (resp. dual of the locally convex space $ C_{\rm c}(U) $).
%, where ${}C_0{}(U) $ denotes the closure of the space $ {}C_c{}(U) $ of all continuous functions with compact supports, in the topology of $ C(\overline{U}) $. 
%As well as, for any $ 1 < \ell \in \mathbb{N} $, we denote by $ \mathcal{M}(U, \mathbb{R}^\ell) $ (resp. $ \mathcal{M}_{\mathrm{loc}}(U, \mathbb{R}^\ell) $) the space of $ \mathbb{R}^\ell $-valued finite Radon measures (resp. $ \mathbb{R}^\ell $-valued Radon measures). 
\medskip

A function $ u\in L^1(U) $ is called a $BV$-function (resp. $BV_\mathrm{loc}$-function) on $ U $ if and only if its distributional gradient $ Du$ is a finite Radon measure (resp. a Radon measure) on $ U $, namely $ Du \in \mathcal{M}(U, \mathbb{R}^d) $ (resp. $ Du \in \mathcal{M}_\mathrm{loc}(U, \mathbb{R}^d) $), and we denote by $ BV(U) $ (resp. $ BV_\mathrm{loc}(U) $) the space of all $BV$-functions (resp. $BV_\mathrm{loc}$-functions) on $ U $. For any $ u\in BV(U) $, the total variation measure $ |Du| \in \mathcal{M}(U) $ of the gradient $ Du $ is called \emph{the total variation measure of $ u $}, Then, by \cite[Proposition 3.6]{MR1857292}, we have, %value $|Du|(U)$ can be calculated that: 
\begin{equation*}
|Du|(U)= \sup\left\{\begin{array}{l|l}\displaystyle\int_U u\,\di  \varphi \,dx & \parbox{6cm}{$ \varphi \in C_c^1(U, \mathbb{R}^d) $ \ and \ $ |\varphi| \leq 1$ on $ U $}\end{array}\right\},
\end{equation*}
and we also write %this quantity is denoted by 
$ \int_U |Du| $ for $|Du|(U)$. 
\medskip

As a function space, $ BV(U) $ is a Banach space, endowed with the norm:
\begin{equation*}
\|u\|_{BV(U)}:= \|u\|_{L^1(U)} + |Du|(U), \mbox{ for any $ u\in BV(U)$}.
\end{equation*}
%Besides, for every $ u \in BV(U) $ and $ \{ u_n \}_{n = 1}^\infty \subset BV(U) $, we say that $ u_n \to u  $ strictly in $ BV(U) $ as $ n \to \infty $, if and only if $ u_n \to u $ in $ L^1(U)$ and $ \int_U |D u_n| \to \int_U |Du| $, as $ n \to \infty$.\medskip

For any $ u \in BV(U) $, we denote by $ {D u^a} $ (respectively, $ \ryota{D u^s} $), the absolutely continuous part (respectively, the singular part of $ Du $) with respect to $ \mathcal{L}^d $. 
%According to \cite[Theorem 3.83 and Proposition 3.92]{MR1857292}, any $ u \in BV(U) $ admits a vector function $ \nabla u \in L^1(U, \mathbb{R}^d) $, called the \emph{approximate differential} of $ u  $, such that:
%\begin{equation*}
%\lim_{\rho \downarrow 0} \frac{1}{\mathcal{L}^d(\rho \mathbb{B}^d)} \int_{x +\rho \mathbb{B}^d} \frac{|u(y) -u(x) -\nabla u(x) \cdot (y -x)|}{\rho} \, dy, \mbox{ a.e. $ x \in U $,}
%\end{equation*}
%and the approximate differential $ \nabla u $ coincides with the density $ \frac{Du}{\mathcal{L}^d} $ a.e. in $ U $. 
Consequently, one can observe that:
\begin{equation*}
Du =\ryota{D u^a} +\ryota{D u^s} = \nabla u \, \mathcal{L}^d +{\textstyle \frac{Du^s}{|Du^s|}} \ryota{|D u^s|}\quad \mbox{ in $ \mathcal{M}(U) $.}
\end{equation*}
Here, $ \frac{Du^s}{|Du^s|} $ denotes the Radon--Nikod\'ym derivative of $ Du^s $ with respect to the total variation measure $ |Du^s| $, and $ \nabla u $ is the approximate differential of $ u \in BV(U) $ (cf. \cite[Definition 3.70]{MR1857292}).

%Furthermore, if the boundary $ \partial U $ of the open set $ U \subset \mathbb{R}^d $ is Lipschitz, then the space $ BV(U) $ is continuously embedded into $ L^{d/(d-1)}(U) $ and compactly embedded into $ L^p(U) $ for any $ 1\le p <d/(d-1) $ (cf. \cite[Corollary 3.49]{MR1857292}, \cite[Theorems 10.1.3--10.1.4]{MR2192832}). 
%Besides, t
There exists a unique linear operator $\gamma_{\partial U} :BV(U) \to L^1(\partial U)$, called \emph{trace operator} such that $ \gamma_{\partial U} \varphi = \varphi |_{\partial U}$ on $\partial U$ for any $u\in C^1(\bar{U})$. %We simply denote by $ \gamma $ the trace operator $ \gamma_{\partial U} $ if there is no possibility of confusion. 

%\begin{remark}\label{Rem.BV}(cf. \cite[Theorem 3.88]{MR1857292}) Let $ U \subset \mathbb{R}^d$ be a bounded domain with a Lipschitz boundary $\partial U$, and let $\nu_{\partial U}$ be the unit outer normal to $\partial U$. Then, it holds that:
%\begin{equation*}

%In addition, the trace $\gamma_{\partial U} : BV(U) \to L^1(\partial U) $ is continuous with respect to the strict topology of $ BV(U) $ (cf. \cite[Theorem 3.88]{MR1857292} and \cite[Theorem 10.22]{MR2192832}). Namely, the convergence of continuous dependence holds:
%\begin{equation}\label{trace}
%\gamma_{\partial U} u_n \to \gamma_{\partial U} u, \mbox{as $ n \to \infty$, for any $u\in BV(U)$ and %$\{u_n\}_{n=1}^\infty \subset BV(U)$,}
%\end{equation}
%in the topology of $L^1(\partial U)$, if $u_n \to u $ strictly in $ BV(U) $ as $ n \to \infty $. 
%However, in contrast with the trace on Sobolev spaces, the convergence \eqref{trace} is not guaranteed, if $ u_n \to u$ weakly-$*$ in $BV(U) $, and even if we adopt any weak topology for \eqref{trace} (including the distributional one).

\paragraph{Notations for the variational analysis (cf. \cite{MR0750538}).}

Throughout this paper, let $ N \in \mathbb{N} $ be a fixed dimension, let $ \Omega \subset \mathbb{R}^N $ be a bounded domain, and let $ \Gamma \subset \mathbb{R}^N $ be a subset of the boundary $ \partial \Omega $ which possibly coincides with the whole $ \partial \Omega $. Also, we assume that the boundary $ \partial \Omega $ has a smoothness of $ C^1 $-class, and we simply denote by $ \nu : \partial \Omega \to \mathbb{S}^{N -1} $ 
the unit outer normal on $ \partial \Omega $, when $ N > 1 $
and $ \gamma : BV(\Omega) \to L^1(\partial \Omega) $  is the trace onto $ \partial \Omega $. On this basis, we define:
\begin{equation}\label{def-Xp}
\begin{array}{c}
L_\mathrm{div}^p(\Omega, \mathbb{R}^N) := \left\{ \begin{array}{r|l}
\omega \in L^p(\Omega, \mathbb{R}^N) & \di \omega \in L^p(\Omega)
\end{array} \right\},\nonumber
\\[1ex]
\mbox{and \ } X_p(\Omega) := L_\mathrm{div}^p(\Omega, \mathbb{R}^N) \cap L^\infty(\Omega, \mathbb{R}^N), 
\mbox{ for any $ 1 \leq p < \infty $.}
\end{array}
\end{equation}
Also, referring to the general theory for $BV$-functions, e.g. \cite[Sections 1--2]{MR0750538}, we recall the following facts.
%\begin{description}
%\item[(Fact 1)]\textbf{(cf. \cite[Sections 1--2]{MR0750538})} 

There exists a bounded linear operator $ \ryota{[({}\cdot{}) \cdot \nu]} : X_2(\Omega) \to L^\infty(\partial \Omega)$, such that 
\begin{equation}\label{f1-1}
\left\{\parbox{10cm}{
 $ \bigl\|\ryota{[\omega \cdot \nu]}\bigr\|_{\infty} \le \| \omega \|_{\infty} $ for any  $\omega \in X_2(\Omega) $,
\\[1ex]
$\ryota{[\tilde{\omega} \cdot \nu]} = \tilde{\omega} \cdot \nu $ on $\partial \Omega$, if $ \tilde{\omega} \in C^1(\bar{\Omega}, \mathbb{R}^N)$.
}\right.
\end{equation}
Besides, for every $ \omega \in X_2(\Omega) $ and $ u\in BV(\Omega) \cap L^2(\Omega)$, there exists a finite Radon measure $(\omega,Du) \in \mathcal{M}(\Omega) $, such that $ (\omega, Du) $ is absolutely continuous for $ |Du| $,   
\begin{equation}\label{f1-2}
    \bigl| {\textstyle \frac{(\omega, Du)}{|Du|}} \bigr| \leq \| \omega \|_{\infty}, \mbox{ $ |Du| $-a.e. in $ \Omega $,}
\end{equation}
and
\begin{equation}\label{f1-3}
\int_\Omega (\omega,Du) = -\int_\Omega \di \omega \, u\,dx + \int_{\partial \Omega} \ryota{[\omega \cdot \nu]} \, \gamma u \,d\mathcal{H}^{N -1}.
\end{equation}
Moreover, for the absolutely continuous part $\ryota{(\omega, D u)^a} $ of $ (\omega, Du) $ for $ \mathcal{L}^N $ and the singular part $ \ryota{(\omega, D u)^s} $, it follows that:
\begin{equation}\label{f1-4}
(\omega, Du) = \ryota{(\omega, D u)^a} +\ryota{(\omega, D u)^s} = \omega \cdot \nabla u \, \mathcal{L}^N +{\textstyle \frac{(\omega, Du)}{|Du|}} \, \ryota{|D u^s|} \mbox{ in $ \mathcal{M}(\Omega) $.}
\end{equation}
%\end{description}

%\begin{remark}\label{F1}
%In comparison with other general theories, e.g. \cite[Section 1.4]{MR0372419}, we can figure out that the operator $ \ryota{[({}\cdot{}) \cdot \nu]}$ coincides with the restriction $ \mathcal{B}_{\partial \Omega} |_{X_2(\Omega)} $ of a bounded linear operator $ \mathcal{B}_{\partial \Omega} : L_\mathrm{div}^2(\Omega, \mathbb{R}^N) \to H^{-\frac{1}{2}}(\partial \Omega) $, such that:
%\begin{equation*}
%\begin{array}{c}
%\displaystyle \bigl\langle \mathcal{B}_{\partial \Omega} \, \omega,  \gamma u \bigr\rangle_{{H^{\frac12}(\partial \Omega)}} = \int_\Omega \omega \cdot \nabla u \, dx +\int_\Omega \di \omega \, u\,dx 
%\\[2ex]
%\mbox{for all $ \omega \in L_\mathrm{div}^2(\Omega, \mathbb{R}^N) $ and $ u \in H^1(\Omega) $.}
%\end{array}
%\end{equation*}
%So, in what follows, we identify these operators if there is no possibility of confusion. 
%\end{remark}

\section{Energy and its lower semi-continuity}\label{synergy}
We want to write (\ref{eq1}) as a gradient flow for a suitable energy functional $E$. We choose the following Hilbert space $H = L^2(\Omega) \times L^2(\Gamma)$ with the standard inner product, 
$$
((u_1, v_1),(u_2,v_2)) = (u_1, u_2)_\Omega + (v_1, v_2)_\Gamma,
$$
where we write $(u_1, u_2)_\Omega =\int_\Omega u_1 u_2\,dx$ and $(v_1, v_2)_\Gamma =\int_\Gamma v_1 v_2\, d\cH^{N-1}$.
We define a functional $ E : H \longrightarrow [0, \infty] $, by setting:
\begin{equation}\label{2dfE}
 E(u,v) = \left\{
\begin{array}{ll}
    {\ds\int_\Omega |Du| + \int_\Gamma|\gamma u -v|\,d\cH^{N-1}}& \hbox{if } (u,v)\in BV(\Omega)\times L^2(\Gamma),\\[2ex]
% +\infty & \hbox{for } (u,v)\in (L^2(\Omega) \setminus BV(\Omega))\times L^2(\Gamma).
    +\infty & \hbox{otherwise. }
\end{array}
 \right. 
\end{equation}
\begin{remark}
We could consider a more general, one-homogeneous function $\Phi$ in place of $|\cdot|$ above. However, this would create another layer of difficulty obscuring the main issue. On the other hand feasibility of such approach is suggested by Moll's paper \cite{MR2139257}. %\cite{moll}.
\end{remark}

The first step is %Before we check the claim we have 
to show that, $E$ defined above, is lower semi-continuous in the $L^2$ topology.

\begin{prop}\label{Prop3.1}
 Let us suppose that $u\in BV(\Omega)$, $v\in L^2(\Gamma)$ and $\{(u_n,v_n)\}_{n=0}^\infty\subset H$ is any sequence converging to $(u,v)$ in $H$. Then,
 $$
 \varliminf_{n\to\infty} E(u_n,v_n) \ge E(u,v).
 $$
\end{prop}
\begin{proof}
If $\Gamma=\d\Omega$, then this fact is well-known, see \cite{soucek}. However, %Since 
the definition of $E$ includes integration over $\Gamma$, which may be essentially smaller than $\d\Omega$, thus we prefer to include the proof. We use here %is based on 
the idea of Giaquinta-Modica-Sou\v cek, \cite{soucek}, to extend the functional $\int_\Omega |Du|$  to a bigger domain. We proceed by taking  any region $\tilde \Omega$ with Lipschitz boundary and such that the following conditions hold:\\
1) $\Omega\subset \tilde \Omega$;\\
2) $\d\tilde \Omega \cap \d \Omega = \d\Omega \setminus \Gamma$;\\
3) the region $\tilde\Omega \setminus\bar  \Omega$ has a Lipschitz continuous boundary.

When $\phi\in L^1(\d\Omega)$ is given, then we may find $\tilde\phi \in W^{1,1}(\tilde\Omega\setminus\bar \Omega)$ such that $\gamma \tilde\phi = \phi$ on $\Gamma$, see \cite{MR0750538, DD}. Then,  we define the following space,
$$
BV_{\Gamma, \phi}(\tilde\Omega) = \{u \in BV(\tilde\Omega): \ u = \tilde\phi \hbox{ on } \tilde\Omega \setminus \bar{\Omega}\}.
$$
It is a well-known fact that functional $BV(\tilde\Omega) \ni u \mapsto \int_{\tilde\Omega} |Du|$ is lower semi-continuous with respect to the $L^2$. As a result, this functional is lower semi-continuous on $BV_{\Gamma, \phi}(\tilde\Omega)$, a closed subspace of  $BV(\tilde\Omega)$. Once we realize that for $u\in BV_{\Gamma, \phi}(\tilde\Omega)$, we have
$$
D u = D u \llcorner \Omega + D u \llcorner \Gamma + \nabla u \llcorner \tilde\Omega \setminus \bar \Omega,
$$
where $D u \llcorner \Gamma = \nu(\phi - \gamma u) \cH^{N-1}\llcorner\Gamma$, where $\gamma u$ is the trace of $u\in BV(\Omega)$,
then
%\todo{I think we do not need to distinguish the trace $ \gamma $ with the restriction $ \gamma_\Gamma $.}
$$
\int_{\tilde\Omega} |D u| = \int_{\Omega} |D u| + \int_{\Gamma} |\phi - \gamma u|\,d \cH^{N-1} +
\int_{\tilde{\Omega} \setminus \bar{\Omega}} |\nabla \tilde\phi|\,dx.
$$
As a result, the functional  
$L^2(\Omega)\ni u\mapsto E(u,v)=:E_v(u)$ is lower semi-continuous. In order to complete the task, we have to consider
$\varliminf_{n\to \infty} E(u_n,v_n),$
when $(u_n,v_n) \to (u,v)$ in $H$. Since $|\gamma u_n - v_n| + | v_n - v| \ge |\gamma u_n - v|$, then we see,
$$
\varliminf_{n\to \infty} E(u_n,v_n) \ge 
\varliminf_{n\to \infty} E_v(u_n) - \lim_{n\to \infty} \int_\Gamma | v_n - v|\,d\cH^{N-1}.
$$
Finally,
our claim follows.
\end{proof}

\begin{remark}\label{Rem.lsc}%\marginpar{re-think}
We noticed in the course of the proof above that for a
fixed $v\in L^2(\Gamma)$,  functional  
%$L^2(\Omega)\ni u\mapsto E(u,v)=:
$E_v(u)$ is lower semi-continuous.
It is %We  notice that $E_v$ is 
a relaxation, i.e. the lower semi-continuous envelope, of the following functional 
$$
% \KSh{E_v^\infty: u \in 
L^2(\Omega)\ni u \mapsto  E^\infty_v(u) = \left\{
\begin{array}{ll}
    {\ds\int_\Omega |Du|}&\hbox{if } u\in BV(\Omega),\ \gamma|_\Gamma u =v,\\[2ex]
 %+\infty & \hbox{for } u\in L^2(\Omega) \setminus BV(\Omega).
    +\infty & \hbox{otherwise.}
\end{array}
 \right.
 $$
 Then,  for any $a>1$ functional $E^a_v : L^2(\Omega) \longrightarrow [0, \infty]$,  given by
 $$
% E^a_v(u) = E(u,v) + a\int_\Gamma |\gamma u -v|\,d\cH^{N-1}
 E^a_v(u) =  \left\{
\begin{array}{ll}
\ds\int_\Omega |Du| +a \int_{\Gamma} |\gamma u -v| \, d \mathcal{H}^{N -1} &\hbox{if } u\in BV(\Omega),\\[2ex]
 %+\infty & \hbox{for } u\in L^2(\Omega) \setminus BV(\Omega).
    +\infty & \hbox{otherwise,}
\end{array}
 \right.
 $$
is not lower semi-continuous, because 
%it is larger than the relation of $E^\infty_v$. 
$ E_v < E_v^a \leq E_v^\infty $ on $ L^2(\Omega) $.
The relaxation of $E^a_v$ is $E_v.$ On the other hand, it is easy to check that for any $a\in (0,1)$ functional $E^a_v$ is lower semi-continuous. Indeed, in this case $E^a_v = a E_v +(1-a) \int_\Omega |Du|$ and both ingredients are lower semi-continuous.
\end{remark}

\begin{remark}\label{Rem.weak.lsc}
We recall that lower semi-continuity of $E$ combined with its convexity implies sequential weak lower semi-continuity.
\end{remark}

\section{The evolution problem and the Comparison Principle} \label{sevolt}
%Brezis book to exploit $E(u\wedge v) + E(u \vee v) = E( u) + E(v).$

We recall two basic abstract facts from the theory of maximal monotone operators.

%
%\begin{thm}[Well-posedness] \label{Prop4.1}
%Let $\mathcal{E}$ be a proper, lower semi-continuous, convex function in a Hilbert space $H$ with values in $\mathbb{R}\cup\{\infty\}$.
%Then, for any $w_0 \in \overline{D(\mathcal{E})}$ there is a unique solution $u \in C\left([0,\infty),H\right) \cap W^{1,1}_{loc}(0,\infty;H)$ satisfying 
%\[
%\frac{du}{dt}(t) \in -\partial\mathcal{E} \left(u(t)\right)
%\quad\text{a.e.}\quad t \in (0,\infty), \quad
%u(0) = w_0.
%\]
%Moreover, $u(t)$ is right-differentiable at all $t>0$ and its right derivative satisfies
%\[
%\frac{d^+ u}{dt}(t) = -\partial^0 \mathcal{E} \left(u(t)\right), \quad
%t>0,
%\]
%where $\partial^0 \mathcal{E}$ denotes the minimal section of $\partial \mathcal{E}$.
%\end{thm}
%\KSh{
\begin{thm}[Well-posedness] \label{Prop4.1}
    Let $\mathcal{E} : X \longrightarrow [0, \infty] $ be a proper, lower semi-continuous, convex function on a Hilbert space $X$.
    Then, for any $w_0 \in D(\mathcal{E})$ there is a unique solution $w \in W^{1, 2}_{loc}([0,\infty);X)$, such that 
\[
\frac{dw}{dt}(t) \in -\partial\mathcal{E} \left(w(t)\right)
    \text{, \quad a.e.}\quad t \in (0,\infty), \quad \mbox{with} \quad
    w(0) = w_0,\quad\mbox{in $ X $.}
\]
Also, the function $ [0, \infty) \ni t \mapsto \mathcal{E}(w(t)) \in [0, \infty) $ is absolutely continuous on any compact interval, and it satisfies that
\begin{equation*}
    \frac{d}{dt} \mathcal{E}(w(t)) = -\left\| \frac{dw}{dt}(t) \right\|_X^2 \quad \mbox{a.e. $ t > 0 $.}
%    \int_{s}^{t} \left| \frac{dw}{dt}(\tau) \right|_X^2 \, d \tau +\mathcal{E}(w(t)) \leq E(w(s)) \mbox{ ~~ for all $ 0 \leq s \leq t < \infty $.}
\end{equation*}
    In particular, if $ w_0 \in D(\d\mathcal{E}) $, then $ w \in W_{loc}^{1, \infty}([0, \infty); X) $, $w$ is right-differentiable over $ [0, \infty) $, and at every $ t \geq 0 $, the right derivative $ \frac{d^+w}{dt}(t) $ satisfies
\[
    \frac{d^+ w}{dt}(t) = -\partial^o \mathcal{E} \left(u(t)\right) \quad \mbox{in $ X $,}
\]
where $\partial^o \mathcal{E}$ denotes the minimal section of $\partial \mathcal{E}$. 
\end{thm}
%}

This type of well-posedness of the gradient flow of a convex functional goes back to K\=omura.
The above version can be found in Br\'ezis, see \cite{MR0348562} or Pazy, see \cite{pazy}.
 Here, we note that $ W_{loc}^{1, 2}([0, \infty); X) $ and $ W_{loc}^{1, \infty}([0, \infty); X) $ are contained in the class $W^{1,1}_{loc}(0,\infty;X)$ of all absolute continuous functions in $(\delta,T)$ for any $T>\delta>0$ with values in the Hilbert space $X$.
\bigskip

In order to proceed, we recall the notion of Banach lattice. 
An ordered Banach space $X$ with ordering $\ge$ is called a {\it vector lattice}, if the linear structure is compatible with the ordering, i.e.
$$
f \ge g \hbox{ implies } f+ h \ge g+ h\quad\hbox{for all }f,g,h\in X;
$$
$$
f \ge 0 \hbox{ implies } \lambda  \ge 0\quad\hbox{for all }f\in X, \lambda \ge 0.
$$
In addition, we require that any two elements $f, g \in X$ have a supremum, denoted by $f\vee g $ and infimum, denoted by $f\wedge g $. Besides,  for $w \in X$, we denote by $w_+$ its positive part, i.e.\ $w_+ = w \vee 0 = \max(w,0)$.  We refer the interested reader for more details on the Banach lattice  to \cite{arendt}.
%\todo{\KSh{Now everything is clear for me. Thank you so much for this remark. KS}}

%Let us suppose that \textcolor{wineRed}a Hilbert space $X$ has a lattice structure.
%%, \textcolor{blue}{given \KSh{by} the relation} $\geq$, }\KSh{ and suppose that it is closed under the operations $ \wedge $, $ \vee $.} Next, w
%We recall the order preserving properties of the gradient flow.
%%Besides,  for $w \in X$, let us denote by $w_+$ its \textcolor{wineRed}{positive} part, i.e.\ $w_+ = w \vee 0 = \max(w,0)$, and
%Let us suppose that %}\todo{Please add a part to explain the meaning of ``lattice structure''.}
%%    \KSh{        
%\begin{equation*}
%\frac{d}{dt} |\omega_+(t)|_X^2 = 2(\omega_+(t),\omega'(t))\quad \mbox{for any $\omega \in W^{1,1}_{loc}(0,\infty;X)$.}
%\end{equation*}
%%    }
%%
\begin{prop}[Order preserving structure] \label{Prop4.2}
%Assume that Hilbert space $X$,  has a lattice structure. %such that 
%\textcolor{blue}{Moreover, it is closed under the operations} $\wedge$,$\vee$.
Assume that a Hilbert space X is a vector lattice. 
Let us suppose that
%\todo{\KSh{I arranged here, in response to (6)--(7) of Prof. Giga's comment. KS}}
\begin{equation*}
\frac{d}{dt} \|\omega_+(t)\|_X^2 = 2(\omega_+(t),\omega'(t))\quad \mbox{for any $\omega \in W^{1,1}_{loc}(0,\infty;X)$.}
\end{equation*}
Let $\mathcal{E}$ in Theorem \ref{Prop4.1} fulfills 
\[
    \mathcal{E}(w_1 \vee w_2) + \mathcal{E}(w_1 \wedge w_2)
    \leq \mathcal{E}({w_1}) + \mathcal{E}({w_2}) \quad\text{for all}\quad
    w_1, w_2 \in D(\mathcal{E}).
\]
    If $w_1$ and $w_2$ are two solutions (in the sense of Theorem  \ref{Prop4.1}) of
        \begin{equation}\label{ken01} 
        \frac{dw}{dt}(t) \in -\partial\mathcal{E}(w(t)) \quad \mbox{ in $ X $,} \quad \mbox{a.e. $ t > 0 $,}
        \end{equation}
    and if the initial data $w_{10}$ and $w_{20}$ satisfy
    \begin{equation*}
        w_1(0) = w_{10} \leq w_{20} = w_2(0) \quad \mbox{in $ X $,}
    \end{equation*}
%respectively,  
then 
    \begin{equation*}
        w_1(t) \leq w_2(t) \quad \mbox{in $ X $,} \quad \mbox{for all $t>0$.}
    \end{equation*}
\end{prop}
This type of argument is well-known. For example it is presented in the thesis of Br\'ezis \cite{brezis-thesis} and more generally in Kenmochi-Mizuta-Nagai, see  \cite{Kenmochi}. We give here a proof since it is elementary.
\begin{proof}
By definition, we see that for a.e.\ $t>0$
\begin{align*}
    \mathcal{E}({\varphi}) - \mathcal{E}\left({w_1(t)}\right) &\geq {\left(w_1'(t),w_1(t)-\varphi\right)_X}
    \quad\text{for all}\quad {\varphi \in X,} \\
    \mathcal{E}({\hat{\varphi}}) - \mathcal{E}\left({w_2(t)}\right) &\geq {\left(w_2'(t),w_2(t)-\hat{\varphi}\right)_X}
    \quad\text{for all}\quad {\hat{\varphi} \in X.}
\end{align*}
In these inequalities, we take
\[
    \begin{cases}
        \varphi = w_1(t) + (w_2-w_1)_+(t) = (w_1 \vee w_2)(t),
        \\
        \hat{\varphi} = w_2(t)-(w_2-w_1)_+(t) = (w_1 \wedge w_2)(t);
    \end{cases}
\]
the last identities follow 
from the property of a vector lattice. Then one gets
%\todo{\KSh{I aranged here in response to (8) of Prof. Giga's comment. KS}}
%where $f_+$ denotes the plus part, i.e., $f_+=\max(f,0)$.
%Then, one gets
%
\begin{align*}
    \mathcal{E} \bigl((w_1 \vee w_2)(t)\bigr) - \mathcal{E}(w_1(t)) &\geq \left(w_1'(t),-(w_2-w_1)_+(t) \right)_X ,\\
    \mathcal{E}\bigl( (w_1 \wedge w_2)(t) \bigr) - \mathcal{E}(w_2(t)) &\geq \left(w_2'(t),(w_2 -w_1)_+(t) \right)_X
\end{align*}
for a.e.\ $t>0$.
Adding these two inequalities and invoking our assumption for $\mathcal{E}$ with respect to $\wedge$ and $\vee$, we see that
\[
  0 \geq \left((w_2'-w_1')(t),(w_2-w_1)_+(t) \right)_X
    = \frac{1}{2} \frac{d}{dt} \left\|(w_2-w_1)_+(t)\right\|_X^2 \quad \mbox{a.e. $ t > 0 $.}
\]
We thus conclude that
        \[
    \frac{d}{dt} \left\|(w_2-w_1)_+(t)\right\|_X^2 \leq 0 \quad \mbox{a.e. $ t > 0 $.}
\]
Thus, the order preserving property follows.
\end{proof}
\begin{remark}
From the proof above, it is easy to claim a comparison principle saying that if $w_1$ is a subsolution and $w_2$ is a supersolution of \eqref{ken01}, then $w_1 \leq w_2$ provided that $w_{1}(0) \leq w_{2}(0)$.
Here, we say $w \in W^{1,2}_{loc}([0,\infty);X)$ is a \textit{subsolution} if for a.e.\ $t>0$ the inequality
\[
    \mathcal{E}\left(w(t)+h\right) - \mathcal{E}\left(w(t)\right) \geq \left(-w'(t),h\right)_X \quad\text{for all } h \in X\text{ and $ h \geq 0 $}
\]
is fulfilled.
 A supersolution is defined in a symmetric way.
\end{remark}

We now consider the gradient flow of $E$ defined in (\ref{2dfE}) 
in a Hilbert space $H=L^2(\Omega)\times L^2(\Gamma)$ equipped with an inner product 
\[
\left((f_1,f_2),(g_1,g_2)\right)_\tau 
= \int_\Omega f_1 g_1 ,dx + \tau \int_\Gamma f_2 g_2 \, d \mathcal{H}^{N-1}
\]
for $f=(f_1,f_2)$, $g=(g_1,g_2)\in H$.
Here, $\tau>0$ is a fixed parameter.
The topology defined by the inner product $(\cdot ,\cdot )_\tau$ is the same but its gradient flow is different.
Formally, the gradient flow with respect to the $(\cdot,\cdot)_\tau$ inner product reads as eq. (\ref{eq1}).
%\begin{align*}
%u_t &= \di \left(\nabla u/|\nabla u|\right), &\hbox{in }\Omega,\\
%\tau v_t &= -\partial u/\partial\mathbf{n}, &\hbox{on }\Gamma.
%\end{align*}
%
Since it is clear that $E$ in (\ref{2dfE}) is convex, lower semi-continuous with respect to the convergence in the standard inner product as well as with respect to $(\cdot,\cdot)_\tau$ and $H=\overline{D(E)}$, Proposition \ref{Prop3.1} enables us to apply Theorem \ref{Prop4.1} to get a well-posedness result.
%
%\begin{thm} \label{Thm4.4}
%For \textcolor{blue}{$U_0:=(u_0,v_0)\in H$, where $H=l^(\Omega)\times L^2(\Gamma)$ and $\cE(U):= E(u,v)$, where $E$ is defined in  (\ref{2dfE}) there is a unique solution $(u,v) \in C\left([0,\infty),H\right) \cap W^{1,1}_{loc} \left(0,\infty; H\right)$ to the following problem} 
%\begin{equation}\label{rn1}
% \begin{array}{ll}
% U_t \equiv (u_t,v_t) \in -\partial_\tau E\left(u(t),v(t)\right)\equiv -\partial_\tau E(U(t)) ,&\text{for a.e.}\quad t>0,\\
%  u(0) = u_0, & v(0) = v_0.
% \end{array}
%\end{equation}
%
%Moreover, $\left(u(t),v(t)\right)$ is right-differentiable at all $t>0$ and
%\[
%\frac{d^+ }{dt}(u,v) = -\partial^0_\tau E \left(u(t),v(t)\right).
%\]
%Here, $\partial_\tau E$ denotes the subdifferential with respect to the inner product $(\cdot ,\cdot )_\tau$.
%\end{thm}
%
\begin{thm} \label{Thm4.4}
    For any $U_0:=(u_0,v_0)\in D(E) = (BV(\Omega) \cap L^2(\Omega)) \times L^2(\Gamma) $ there is a unique solution $ U := (u,v) \in W^{1,2}_{loc} \left([0,\infty); H\right)$ of the following problem 
\begin{align}\label{rn1}
    ~& \ds \frac{dU}{dt}(t) := \frac{d}{dt} \left( u(t), v(t) \right) 
%    \\
%    &~ 
    \in -\partial_\tau E(U(t))\equiv -\partial_\tau E\left(u(t),v(t)\right) \quad
    %%\nonumber
    %%\\
    %%& \hspace{20ex}
    \mbox{in $ H $,} \quad \text{a.e. } t>0,\\ %%[1ex]
     ~& U(0) \equiv (u(0), v(0)) = (u_0, v_0) \quad \mbox{in $ H $,} 
     \nonumber
\end{align}
where $\partial_\tau E$ denotes the subdifferential of $ E $ with respect to the inner product $(\cdot ,\cdot )_\tau$. Also, the function $ t \in [0, \infty) \longrightarrow E(U(t)) \equiv E(u(t), v(t)) \in [0, \infty) $ is absolutely continuous on any compact interval, and it satisfies that
    \begin{align}\label{rn2}
        \frac{d}{dt} E(U(t)) &~ = -\left( \frac{dU}{dt}(t), \frac{dU}{dt}(t)\right)_\tau
        \nonumber
        \\
        &~ =  -\left\| \frac{du}{dt}(t) \right\|_{L^2(\Omega)}^2 \hspace{-0ex} -\tau \left\| \frac{dv}{dt}(t) \right\|_{L^2(\Gamma)}^2\quad \mbox{a.e. $ t > 0 $.}
\end{align}
    In particular, if $ U_0 \in D(\partial_\tau E) $, then  $ U = (u, v) \in W_{loc}^{1, \infty}([0, \infty); H) $, $ U $ is right-differentiable over $ [0, \infty) $, and at every $t>0$, the right derivative $ \frac{d^+U}{dt}(t) := \frac{d^+}{dt}(u(t), v(t)) $ satisfies
    \[
    \frac{d^+ }{dt}U(t) = -\partial^o_\tau E \left(U(t)\right) \quad \mbox{in $ H $,}\quad \mbox{a.e. $ t > 0 $.}\eqno\qed
\]
\end{thm}

We notice that $H$ has the desired lattice structure after we define
\[
(f_1,f_2) \leq (g_1,g_2) \quad\text{for}\quad
f = (f_1,f_2),\ g = (g_1,g_2) \in H 
\]
if $f_1 \leq f_2$ a.e.\ in $\Omega$ and $g_1 \leq g_2$\ $\mathcal{H}^{N-1}$-a.e.\ on $\Gamma$. We  also check that the functional $E$ has the desired properties:
%\pagebreak
\begin{prop} If $E$ is defined by formula (\ref{2dfE}), then
    \[
E(u_1 \wedge u_2, v_1 \wedge v_2) + E(u_1 \vee u_2, v_1 \vee v_2)
    \leq E(u_1, v_1) + E(u_2, v_2).
\]
%\todo{\textcolor{wineRed}{This corrected inequality does not interfere with the notion of ``lattice structure'', PR}}
%\todo{\KSh{I understood, thank you again. KS}}
\end{prop}
\begin{proof}
First, by referring to \cite[Lemmas 2.2 and 3.1]{MR2101878}, we verify that
%We recall that
%\KSh{
%$$
%    \begin{array}{c}
%    w_1 \vee w_2 = w_1 +(w_2 -w_1)_+ \quad \mbox{and} \quad w_1 \wedge w_2 = w_2 -(w_2 -w_1)_+.
%        \\[1ex]
%        \mbox{for all real functions $ w_i $, $ i = 1, 2 $, defined on the same domain.}
%    \end{array}
%$$
%}
%Keeping this in mind we see that
%\begin{eqnarray*}
% \int_\Omega |D( u_1 \vee u_2)| &+ &\int_\Omega |D( u_1 \wedge u_2)| \\&=&
%\int_{\{u_2\le u_1\}} |Du_1| + \int_{\{u_2>u_1\}} |Du_2| + 
%\int_{\{u_1\le u_2\}} |Du_2| +\int_{\{u_2>u_1\}} |Du_1| \\&=&
%\int_\Omega |Du_1| + \int_\Omega |D u_2|.
%\end{eqnarray*}
    \begin{equation*}
        \int_\Omega |D(u_1 \vee u_2)| +\int_\Omega |D(u_1 \wedge u_2)| \leq \int_\Omega |D u_1| +\int_\Omega |D u_2|.
    \end{equation*}
%\todo{I think the previous verification of the equality is valid only in smooth cases of functions. In fact, the sets $ \{ u_2 \leq u_1 \} $ and $ \{ u_2 > u_1 \} $ are considered in the scope of Lebesgue measurable sets. }\\
We also have to show that
    $$
\int_\Gamma (|u_1 \vee u_2 - v_1 \vee v_2| +
|u_1 \wedge u_2 - v_1 \wedge v_2|) \,d\cH^{N-1} =
\int_\Gamma (|u_1  - v_1 | +
| u_2 -  v_2|) \,d\cH^{N-1} 
$$
where we identified $u_i$, $i=1,2$ with their traces on $\Gamma$.

Since the roles of $u_1$ and $u_2$ are interchangeable, we may assume that 
$u_1 \vee u_2 = u_1$ and $u_1 \wedge u_2 = u_2$. If $v_1 \ge v_2$, then there is nothing to prove, thus we may assume that $v_1 < v_2$. Finally, we have to check that
$$
|u_1 - v_2| + |u_2 - v_1 | = |u_1  - v_1 | + | u_2 -  v_2|\quad \hbox{for a.e. }x\in \Gamma.
$$
However, this obviously holds for all $u_2$.
\end{proof}

The result we have just proved permits us 
%we are able 
to apply Proposition \ref{Prop4.2} to conclude the order preserving property.
\begin{thm}[Order preserving property] \label{Thm4.5}
Let $U_i = (u_i,v_i)$, $i=1,2,$ be a solution in Theorem \ref{Thm4.4} starting from $U_{i_0}=(u_{i_0},v_{i_0})\in H$.
If $U_{10}\leq U_{20}$, then $U_1(t) \leq U_2(t)$ for all $t>0$, i.e., $u_1(t)\leq u_2(t)$, $v_1(t) \leq v_2(t)$ for all $t>0$.
\end{thm}

\section{The subdifferential and its canonical section}\label{ssubd}

This section is devoted to the characterizations of the subdifferential of $ E(u, v) $ given by \eqref{2dfE} and its canonical section. Even though eq. (\ref{rn1}) contains a parameter $\tau>0$, we shall see that without the loss of generality, it is sufficient to calculate the subdifferential of $ E(u, v) $ with respect to the standard inner product of $H$.
%\textcolor{blue}{Since we introduced in the previous section the inner product in $H$ with the parameter $\tau$ we will continue to work with this new structure.}
% These issues are discussed in two subsections.  
%The goal is to calculate the subdifferential of $E(u,v)$ given by (\ref{2dfE}).
%For this purpose we use results of the characterization of the subdifferential of $\Phi$ in the $L^2$, given by,
%$$
%\Phi(u) = \int_\Omega |Du| + \int_{\d\Omega}|u-\varphi|\,\mathrm{d} \cH^{N-1},
%$$
%where $\varphi$ is a given element, see \cite[Proposition 5.10, Lemma 5.13]{mazon}. Namely,
%\begin{equation}\label{subd}
%\d\Phi(u) = \{ v\in L^2(\Omega): \ v = - div z, \ \|z\|_\infty \le 1, \ z\cdot \nu \in sgn (\varphi -u)\}.
%\end{equation}
%We consider $E$ defined on $L^2(\Omega)\times L^2(\d\Omega)$. First, we define an operator $\cA\subset H\times H$. Namely,
%$(U,Z) \in \cA$ if and only if the following conditions are satisfied:\\
%$U=(u,v)$, $Z=(\xi,\zeta)$, $u\in BV(\Omega)$, $v\in L^2(\Gamma)$ and there exists $z\in L^\infty(\Omega,\bR^N)$, $\di z\in L^2(\Omega)$ such that\\
%(1) $\xi = - \di z$;\\
%(2) $z(x) \in \d | \nabla u(x) |$ for $\cL^N$- a.e. $x\in \Omega$;\\
%(3) $z \cdot Du^s = | Du^s |$;\\
%(4) $[z\cdot\nu]+ \zeta =0$ on $\Gamma$;\\
%(5) $\zeta \in \Sgn(u-v)$;\\
%(6) $[z\cdot\nu] =0$ on $\d\Omega\Gamma$.
%
%\bigskip
%We
%claim:
%
%\begin{prop}
%$$
%\cA = \d E(u,v).
%$$
%\end{prop}

\subsection{The representation of the subdifferential}%\marginpar{R.N. + K.Sh.}

The goal of this subsection is to prove the following proposition.

%\todo{I changed Proposition \ref{Prop.rep_subd} to Theorem \ref{Prop.rep_subd}, in accordance with the last our discussion.}
\begin{thm}[Representation of the subdifferential]\label{Prop.rep_subd}
    Let $ E(u, v) $ be given by \eqref{2dfE},  as a result it is a proper, lower semi-continuous and convex function on $ H $. Then, for pairs of functions $ (u, v) \in H $ and $ (\xi,\zeta) \in H $, the following two statements are equivalent.

\begin{description}
\item[\textmd{(A)}]$ (\xi, \zeta) \in \partial E(u, v) $ in $ H $ when $ (u, v) \in D(\partial E) $.
\item[\textmd{(B)}]$ (u, v) \in BV(\Omega) \times L^2(\Gamma) $, and there exists a %bounded 
vector field $ z \in X_2(\Omega) $, such that:
\begin{description}
\item[\textmd{(b1)}]$ \frac{(z, Du)}{|Du|} = 1 $, $ |Du| $-a.e. in $ \Omega $, and moreover, $ z \in \Sgn^N(\nabla u) $ a.e. in $ \Omega $;
\item[\textmd{(b2)}]$ - [z \cdot \nu] \in \Sgn(\gamma u -v) $ a.e. on $ \Gamma $, and $ [z \cdot \nu] = 0 $ a.e. on $ \partial \Omega \setminus \Gamma $; %\marginpar{$\textcolor{blue}{\frac 1\tau}$}
\item[\textmd{(b3)}]$ \xi = -\di z $ in $ L^2(\Omega) $, and $ \zeta =  [z \cdot \nu] $ in $ L^2(\Gamma) $; %\marginpar{$\textcolor{blue}{\frac 1\tau}$}
\end{description}\end{description} 
where $ \Sgn  $ is the abbreviation of the set-valued function $ \Sgn^1 : \mathbb{R} \to 2^{\mathbb{R}} $, %i.e. the signal-function $ \Sgn^d $, as 
defined in Remark \ref{Rem.conv}, when $ d = 1 $. 
\end{thm}

For the proof of this proposition, we first prepare some additional notations with an auxiliary lemma.
\paragraph{\boldmath Operator $ \mathcal{A} $.}
We define a set-valued operator $ \mathcal{A} \subset H \times H $ by letting:
\begin{align}\label{Op.A}
(u, v) \in H \mapsto \mathcal{A}(u, v) := \left\{ \begin{array}{l|l}
(\xi, \zeta) \in H & \parbox{6.5cm}{
$ (\xi, \zeta) $ as in %atisfies 
condition (b3), where %with some 
    $ z \in X_2(\Omega) $ satisfies conditions (b1)--(b2) in Theorem \ref{Prop.rep_subd}
}
\end{array} \right\},
\end{align}
and we denote by $ D(\mathcal{A}) $ the domain of this operator, i.e. 
\begin{equation*}
D(\mathcal{A}) := \left\{ \begin{array}{l|l}
(\tilde{u}, \tilde{v}) \in H & \mathcal{A}(\tilde{u}, \tilde{v}) \ne \emptyset
\end{array} \right\}.
\end{equation*}

\paragraph{\boldmath Relaxed convex function $ E_\varepsilon(u, v) $.}
We define a sequence $ \{ E_\varepsilon \}_{\varepsilon > 0} $ of lower semi-continuous, convex functionals $E_\varepsilon : H \longrightarrow [0, \infty]$, %for $ E(u, v) $, 
by letting for any $ \varepsilon > 0 $,
%For any $ \varepsilon > 0 $, we define a functional  $ E_\varepsilon : H \longrightarrow [0, \infty] $ by putting:
\begin{equation}\label{2dfE_eps}
\begin{array}{rl}
%(u, v) & \in H \mapsto E_\varepsilon(u, v) 
%\\[1ex]

E_\varepsilon (u, v)& := \left\{ \begin{array}{ll}
\multicolumn{2}{l}{\displaystyle \int_\Omega \sqrt{|\nabla u|^2 +\varepsilon^2}\,dx +\frac{\varepsilon^2}{2} \int_\Omega |\nabla u|^2 \, dx,} %\in (0, \infty),}
\\[2ex]
& \mbox{if $ u \in H^1(\Omega) $ and $ \gamma u  = v $ a.e. on $ \Gamma $,}
\\[2ex]
\infty, & \mbox{otherwise\ryota{,}}
\end{array} \right.
%\mbox{\raisebox{-0.75ex}{for all $ \varepsilon > 0 $.}}
\end{array}
\end{equation}
Note that for every $ \varepsilon > 0 $, $ E_\varepsilon $ are proper %and convex 
on $ H $. Also, 
\begin{equation}\label{D(E_eps)}
D(E_\varepsilon) = V := \left\{ \begin{array}{l|l}
(\tilde{u}, \tilde{v}) \in \ryota{H^1(\Omega) \times H^1(\Gamma)} & \parbox{3cm}{
$ \gamma \tilde{u}   = \tilde{v} $ a.e. on $ \Gamma $
}
\end{array} \right\}, \mbox{ for $ \varepsilon > 0 $,}
\end{equation}
namely, the effective domains $ D(E_\varepsilon) $, for $ \varepsilon > 0 $, are equal to %uniformly coincide with 
a closed linear subspace $ V $ in $ H^1(\Omega) \times H^1(\Gamma) $. The equality %coincidence as 
in \eqref{D(E_eps)} is essential to guarantee the lower semi-continuity of the convex functions $ E_\varepsilon $, for $ \varepsilon > 0 $. 

\begin{lemma}\label{axLem01}
Let us fix any constant $ \varepsilon > 0 $, and let us set:
\begin{equation*}
D_\varepsilon := \left\{ \begin{array}{l|l}
(\tilde{u}, \tilde{v}) \in D(E_\varepsilon) & \parbox{7.5cm}{
$ \frac{\nabla \tilde{u}}{\sqrt{|\nabla \tilde{u}|^2 +\varepsilon^2}} +\varepsilon^2 \nabla \tilde{u} \in L_{\mathrm{div}}^2(\Omega, \mathbb{R}^N) $, 
\\[1ex]
$ \bigl[ \bigl( \frac{\nabla \tilde{u}}{\sqrt{|\nabla \tilde{u}|^2 +\varepsilon^2}} +\varepsilon^2 \nabla \tilde{u} \bigr) \cdot \nu \bigr] \in L^2(\partial \Omega) $, and 
\\[1ex]
$ \bigl[ \bigl( \frac{\nabla \tilde{u}}{\sqrt{|\nabla \tilde{u}|^2 +\varepsilon^2}} +\varepsilon^2 \nabla \tilde{u} \bigr) \cdot \nu \bigr] = 0 $ a.e. on $ \partial \Omega \setminus \Gamma $
}
\end{array} \right\}.
\end{equation*}
Then, the subdifferential $ \partial E_\varepsilon \subset H \times H $ coincides with a single-valued operator $ \mathcal{A}_\varepsilon \subset H \times H $, defined as follows:
%\marginpar{$\textcolor{blue}{\frac 1\tau}$}
\begin{equation}\label{A_eps}
(u, v) \in D_\varepsilon \subset H \mapsto \mathcal{A}_\varepsilon (u, v) := \rule{0pt}{24pt}^{\mathrm{t}} \hspace{-1ex} \left( \begin{array}{c}
-\di  \bigl( \frac{\nabla u}{\sqrt{|\nabla u|^2 +\varepsilon^2}} +\varepsilon^2 \nabla u \bigr)
\\[2ex]
 \bigl[ \bigl( \frac{\nabla u}{\sqrt{|\nabla u|^2 +\varepsilon^2}} +\varepsilon^2 \nabla u \bigr) \cdot \nu \bigr]
\end{array} \right) \in H.
\end{equation}
\end{lemma}
\begin{proof}
This lemma can be obtained as a straightforward consequence of standard variational methods (cf.\cite{MR2582280,MR1727362}). 
\end{proof}
\begin{proof}[Proof of Theorem \ref{Prop.rep_subd}]
With the use of the operator $ \mathcal{A} $ given by \eqref{Op.A}, the conclusion of the proposition can be rephrased as follows:
\begin{equation}\label{rep.subd010}
\mathcal{A} = \partial E \mbox{ in $ H \times H $.}
\end{equation}
We check this equality with the help of the %above coincidence via the 
following two Claims $\sharp1$--$\sharp2$.

\paragraph{\textmd{\textit{Claim $\mathit{\sharp1}$: $ {\mathcal{A} \subset \partial E} $ in $ H  \times H $.}}}
Let us assume that:
\begin{equation*}
(\xi, \zeta) \in \mathcal{A}(u, v) \mbox{ in $ H $}  \mbox{ with } (u, v) \in D(\mathcal{A}).
\end{equation*}
Then, in the light of \eqref{f1-3} and \eqref{Op.A}, we can see that for any $ (\varphi, \psi) \in D(E) $ we have:
\begin{align*}
\bigl( (\xi,\, & \zeta), (\varphi, \psi) -(u, v) \bigr)
\\[1ex]
& = \int_\Omega -\di z \, (\varphi -u) \, dx +\int_\Gamma [z \cdot \nu] (\psi -v) \, d \mathcal{H}^{N -1}
\\[1ex]
& = \int_\Omega (z, D(\varphi -u)) -\int_{\partial \Omega} [z \cdot \nu]  \gamma(\varphi -u) \, d\mathcal{H}^{N -1} +\int_\Gamma [z \cdot \nu] (\psi -v) \, d\mathcal{H}^{N -1}
\\[1ex]
& \leq \|z\|_{\infty} \int_\Omega |D \varphi| -\int_\Omega |Du| +\int_\Gamma \bigl( |\gamma \varphi -\psi| -|\gamma u -v| \bigr) \, d \mathcal{H}^{N -1}
\\[1ex]
& \leq E(\varphi, \psi) -E(u, v).
\end{align*}
Thus, $ (\xi, \zeta) \in \partial E(u, v) $ in $ H $. 

\paragraph{\textmd{\textit{Claim $\mathit{\sharp2}$: $ {(\mathcal{A} +\mathcal{I}_H)H = H} $.}}}
Since the inclusion $ (\mathcal{A} +\mathcal{I}_H)H \subset H $ is trivial, our task can be reduced to show only the converse one. 

Let us fix any $ (f, g) \in H $. Then, applying Minty's theorem and Lemma \ref{axLem01}, we can find a sequence of functions $ \{ (u_\varepsilon, v_\varepsilon) \}_{\varepsilon > 0} \subset V $ such that:
\begin{equation}\label{rep_subd020}
(f, g) -(u_\varepsilon, v_\varepsilon) \in \partial E_\varepsilon(u_\varepsilon, v_\varepsilon) \mbox{ in $ H $, for all $ \varepsilon > 0 $.}
\end{equation}
Here, with \eqref{2dfE_eps} and Lemma \ref{axLem01} in mind, we multiply the both sides of \eqref{rep_subd020} by $ (u_\varepsilon, v_\varepsilon) $. Then, by using Young's inequality, one can immediately see that:
\begin{align}\label{rep_subd030}
	\frac{1}{2} \bigl\| (u_\varepsilon, v_\varepsilon) \bigr\|_{H}^2 +E_\varepsilon(u_\varepsilon, v_\varepsilon) & \ryota{\leq \frac{1}{2} \bigl\| (f, g) \bigr\|_{H}^2 +E_\varepsilon(0, 0)}
\nonumber
\\
& \leq \frac{1}{2} \bigl\| (f, g) \bigr\|_{H}^2 +\varepsilon \mathcal{L}^N(\Omega), \mbox{ for all $ \varepsilon > 0 $.}
\end{align}
Subsequently, invoking \eqref{2dfE}, \eqref{2dfE_eps}, \eqref{rep_subd030} and the compactness theorem of Rellich--Kondrashov type, we can find an approximating limit $ (u, v) \in D(E) $ together with a sequence $ \{ \varepsilon_n \}_{n = 1}^\infty \subset (0, 1) $ and a sequence of functions $ \{ (u_n, v_n) \}_{n = 1}^\infty := \{ (u_{\varepsilon_n}, v_{\varepsilon_n}) \}_{n = 1}^{\infty} \subset V $, %a sequence of convex functions $ \{ E_n \}_{n = 1}^{\infty} := \{ E_{\varepsilon_n} : H \longrightarrow [0, \infty] \, | \, n \in \mathbb{N} \} $, 
such that:
\begin{equation}\label{rep_subd040}
1 > \varepsilon_1 > \varepsilon_2 > \cdots > \varepsilon_n \downarrow 0 \mbox{ as $ n \to \infty $,}
\end{equation}
\begin{equation}\label{rep_subd050}
\begin{cases}
(u_n, v_n) \to (u, v) \mbox{ weakly in $ H $,}
\\
u_n \to u \mbox{ in $ L^1(\Omega) $,}
\\
\varepsilon_n u_n \to 0 \mbox{ weakly in $ H^1(\Omega) $,}
\end{cases}
\mbox{as $ n \to \infty $,}
\end{equation}
and for any $ (\varphi, \psi) \in V $, 
%\marginpar{$\textcolor{blue}{\tau}$}
\begin{align}\label{rep_subd060}
\int_\Omega {\textstyle \frac{\nabla u_n}{\sqrt{|\nabla u_n|^2 +\varepsilon_n^2}}} & \cdot \nabla \varphi \, dx +\int_\Omega \nabla (\varepsilon_n u_n) \cdot \nabla (\varepsilon_n \varphi) \, dx 
\nonumber
\\
& \qquad = \int_\Omega (f -u_n) \varphi \, dx +%\textcolor{blue}{\tau}
\int_\Gamma (g -v_n) \psi \, d \mathcal{H}^{N -1},\quad n\in \mathbb{N}.
\end{align}
Additionally, since
\begin{equation*}
{\textstyle
\bigl| \frac{\nabla u_n}{\sqrt{|\nabla u_n|^2 +\varepsilon_n^2}} \bigr| \leq 1 \mbox{ a.e. in $ \Omega $, for $ n = 1, 2, 3, \dots $,}
}
\end{equation*}
we may assume existence of
%suppose the presence of 
a vector field $ z \in L^\infty(\Omega, \mathbb{R}^N) $ such that
\begin{equation}\label{rep_subd070}
\begin{array}{c}
    {\textstyle \frac{\nabla u_n}{\sqrt{|\nabla u_n|^2 +\varepsilon_n^2}}} \to z \mbox{ weakly-$*$ in $ L^\infty(\Omega, \mathbb{R}^N) $ as $ n \to \infty $,}
\\[1ex]
\mbox{ and } |z| \leq 1 \mbox{ a.e. in $ \Omega $,}
\end{array}
\end{equation}
by taking an additional %more 
subsequence, if necessary.\bigskip

Now, applying the convergences \eqref{rep_subd040}--\eqref{rep_subd050} and \eqref{rep_subd070} to the variational form \eqref{rep_subd060}, we can see that for any $ (\varphi, \psi) \in V $, we have: 
\begin{equation}\label{rep_subd080}
\begin{array}{c}
\displaystyle 
\int_\Omega z \cdot \nabla \varphi \, dx = \int_\Omega (f -u) \varphi \, dx +
\int_\Gamma (g -v) \psi \, d \mathcal{H}^{N -1}.
\end{array}
\end{equation}
In particular, taking any $ \varphi_0 \in H_0^1(\Omega) $ and putting $ (\varphi, \psi) = (\varphi_0, 0) \in V $ in  \eqref{rep_subd080}, we have:
\begin{equation*}
\int_\Omega z \cdot \nabla \varphi_0 \, dx = \int_\Omega (f -u) \varphi_0 \, dx, \mbox{ for any $ \varphi_0 \in H_0^1(\Omega) $,}
\end{equation*}
i.e.
\begin{equation}\label{rep_subd090}
-\di z = f -u \in L^2(\Omega) \mbox{ in $ H^{-1}(\Omega) $.}
\end{equation}
Subsequently, for any $ \tilde{\psi} \in H^1(\partial \Omega) $, we invoke \cite[Proposition 5.6.3]{MR2192832} with the $C^1$-smoothness of $ \partial \Omega $ to take an extension $ \tilde{\psi}^\mathrm{ex} \in H^1(\Omega) $ of $ \tilde{\psi} $. Then, putting $ (\varphi, \psi) = (\tilde{\psi}^{\mathrm{ex}}, \tilde{\psi}) \in V $ in \eqref{rep_subd080}, we deduce %is observed
from \eqref{rep_subd090} that 
for any $ \tilde{\psi} \in H^1(\partial \Omega) $, we have,
\begin{equation*}
\begin{array}{c}
\dps\int_{\partial \Omega} [z \cdot \nu] \tilde{\psi} \, d \mathcal{H}^{N -1} =  \int_\Gamma (g -v) \tilde{\psi} \, d \mathcal{H}^{N -1}  =  \int_{\partial \Omega} [g -v]_0^\mathrm{ex} \tilde{\psi} \, d \mathcal{H}^{N -1},
%\\[2ex]
%\mbox{for any $ \tilde{\psi} \in H^1(\partial \Omega) $,}
\end{array}
\end{equation*}
where $ [g -v]_0^\mathrm{ex} \in L^2(\partial \Omega) $ is the zero-extension of $ g -v \in L^2(\Gamma) $. This implies that, %\marginpar{$\textcolor{blue}{\frac1\tau}$}
\begin{equation}\label{rep_subd100}
[z \cdot \nu](y) = \left\{ \begin{array}{l}
 (g -v)(y), \mbox{ if $ y \in \Gamma $,}
\\[1ex]
0, \mbox{ if $ y \in \partial \Omega \setminus \Gamma $,}
\end{array} \right. \mbox{for a.e. $ y \in \partial \Omega $.}
\end{equation}

Finally, we take any $ (\tilde{\varphi}, \tilde{\psi}) \in V $, and put $ (\varphi, \psi) = (u_n -\tilde{\varphi}, v_n -\tilde{\psi}) \in V $ in \eqref{rep_subd060} to obtain that: %\marginpar{$\textcolor{blue}{\tau}$}
\begin{align}
  & E_{\varepsilon_n}(u_n, v_n)  +\int_\Omega (u_n -f)(u_n -\tilde{\varphi}) \, dx + \int_\Gamma (v_n -g)(v_n -\tilde{\psi}) \, d \mathcal{H}^{N -1}
\nonumber
\\[1ex]
& \quad \leq \int_\Omega {\txs \frac{\nabla u_n}{\sqrt{|\nabla u_n|^2 +\varepsilon_n^2}}} \cdot \nabla \tilde{\varphi} \, dx +\frac{\varepsilon_n^2}{2} \int_\omega |\nabla \tilde{\varphi}|^2 \, dx + E_{\varepsilon_n}(0, 0)
\nonumber
\\[1ex]
& \ryota{ \quad = \int_\omega {\txs \frac{\nabla u_n}{\sqrt{|\nabla u_n|^2 +\varepsilon_n^2}}} \cdot \nabla \tilde{\varphi} \, dx +\frac{\varepsilon_n^2}{2} \int_\omega |\nabla \tilde{\varphi}|^2 \, dx +\varepsilon_n \mathcal{L}^n(\omega), \mbox{ for $ n = 1, 2, 3, \dots $.}}
\label{rep_subd120}
\end{align}
Here, having in mind \eqref{2dfE}, \eqref{2dfE_eps}, \eqref{rep_subd040}--\eqref{rep_subd050},  \eqref{rep_subd070} and the weak lower semi-continuity of $ E $ on $ H $, let us take the limit-inf of both sides of \eqref{rep_subd120}. Then, we compute, %that:
\begin{align}
E(u, v) & +\int_\Omega (u -f)(u -\tilde{\varphi}) \, dx + \int_\Gamma (v -g)(v -\tilde{\psi}) \, d \mathcal{H}^{N -1}\nonumber
\\[1ex]
\leq & \varliminf_{n \to \infty} E_{\varepsilon_n}(u_n, v_n) +\varliminf_{n \to \infty} \int_\Omega  (u_n -f)(u_n -\tilde{\varphi}) \, dx%\nonumber \\& \hspace{13.5ex} 
+\varliminf_{n \to \infty}  \int_\Gamma (v_n -g)(v_n -\tilde{\psi}) \, d \mathcal{H}^{N -1}\nonumber
\\[1ex]
\leq & \varliminf_{n \to \infty} \left( E_{\varepsilon_n}(u_n, v_n) -\int_\Omega  (f -u_n)(u_n -\tilde{\varphi}) \, dx %\right.\nonumber\\& \hspace{17ex} \left. 
-\int_\Gamma  (g -v_n)(v_n -\tilde{\psi}) \, d \mathcal{H}^{N -1} \right)\nonumber
\\[1ex]
\leq & \int_\Omega z \cdot \nabla \tilde{\varphi} \, dx.\nonumber%\ryota{\label{rep_subd125}}
\end{align}
\ryota{Therefore, for any $(\tilde{\varphi},\tilde{\psi}) \in V$, we have,
\begin{equation}\label{rep_subd125}
\begin{array}{lc}
\multicolumn{2}{l}{\dps\int_\Omega |Du| + \int_\Gamma |\gamma u - v|\, d\mathcal{H}^{N-1}}
\\[2ex]
& \dps\leq \int_\Omega z \cdot \nabla \tilde{\varphi} \,dx + \int_\Omega (f - u)(u - \tilde{\varphi}) \,dx +  \int_\Gamma (g-v)( v -\tilde{\psi})\, d\mathcal{H}^{N-1}.
%,\\[2ex]& \mbox{for any $(\tilde{\varphi},\tilde{\psi}) \in V$. }  
\end{array}
\end{equation}
}
Additionally, by applying \ryota{\eqref{f1-3}--\eqref{f1-4}}, \eqref{rep_subd090}--\eqref{rep_subd100} \ryota{and \eqref{rep_subd125}}, we can deduce that:
\begin{align}%\label{rep_subd130}
\int_\Omega & |Du|  +\int_\Gamma |\gamma u -v| \, d\mathcal{H}^{N -1}
%\nonumber
%\\
%& 
    = \int_\Omega |\nabla u| \, dx +\int_\Omega \ryota{|D u^s|}  +\int_\Gamma |\gamma u -v| \, d\mathcal{H}^{N -1}
\nonumber
\\
& \ryota{\leq -\int_\Omega \di z \, \tilde{\varphi} \, dx + \int_{\partial \Omega} [z \cdot \nu]\gamma \tilde{\varphi} \,d\mathcal{H}^{N-1}} \nonumber
\\
& \qquad \ryota{+ \int_\Omega (-\di z) (u -\tilde{\varphi}) \, dx +\int_{\Gamma} [z \cdot \nu] (v- \tilde{\psi}) \, d \mathcal{H}^{N -1}} \nonumber
\\
& \ryota{= -\int_\Omega \di z \, u \, dx +\int_\Gamma [z \cdot \nu] v \, d \mathcal{H}^{N -1}}
\nonumber
\\
& = \int_\Omega (z, Du) -\int_{\partial \Omega} [z \cdot \nu] \gamma u \, d \mathcal{H}^{N -1} +\int_\Gamma [z \cdot \nu] v \, d \mathcal{H}^{N -1}
\nonumber
\\
&= \int_\Omega {\textstyle \frac{(z, D u)}{|Du|}} \, |Du| +\int_\Gamma -[z \cdot \nu](\gamma u -v) \, d \mathcal{H}^{N -1}.
\nonumber
\\
&= \int_\Omega z \cdot \nabla u \, dx +\int_\Omega \ryota{(z, D u)^s} +\int_\Gamma -[z \cdot \nu](\gamma u -v) \, d \mathcal{H}^{N -1}. \label{rep_subd130}
\end{align}
%%%% old ver. %%%%%%%%%%%%%%%%%%%%%%%%%%%%%%%%%%%%%%%%%%%%%%%%%%%%%%%%%%%%%%%%%%%%%%%%%%%%
%Additionally, by applying \eqref{f1-3}, \eqref{rep_subd090}--\eqref{rep_subd100}, it is deduced that:
%\begin{align}\label{rep_subd130}
%\int_\Omega & |Du|  +\int_\Gamma |\gamma u -v| \, d\mathcal{H}^{N -1}
%\nonumber
%\\
%& = \int_\Omega |\nabla u| \, dx +\int_\Omega |Du^s|  +\int_\Gamma |\gamma u -v| \, d\mathcal{H}^{N -1}
%\nonumber
%\\
%& \leq -\int_\Omega \di z \, u \, dx 
%+\int_\Gamma [z \cdot \nu] v \, d \mathcal{H}^{N -1}
%\nonumber
%\\
%& \qquad +\int_\Omega z \cdot \nabla \tilde{\varphi} \, dx +\int_\Omega \di z \, \tilde{\varphi} \, dx -\int_{\partial \Omega} [z \cdot \nu] \gamma \tilde{\varphi} \, d \mathcal{H}^{N -1} 
%\\
%& = \int_\Omega (z, Du) -\int_{\partial \Omega} [z \cdot \nu] \gamma u \, d \mathcal{H}^{N -1} +\int_\Gamma [z \cdot \nu] v \, d \mathcal{H}^{N -1}
%\nonumber
%\\
%&= \int_\Omega {\textstyle \frac{(z, D u)}{|Du|}} \, |Du| +\int_\Gamma -[z \cdot \nu](\gamma u -v) \, d \mathcal{H}^{N -1}.
%\nonumber
%\\
%&= \int_\Omega z \cdot \nabla u \, dx +\int_\Omega (z, Du^s) +\int_\Gamma -[z \cdot \nu](\gamma u -v) \, d \mathcal{H}^{N -1}. 
%\end{align}
%%%%%%%%%%%%%%%%%%%%%%%%%%%%%%%%%%%%%%%%%%%%%%%%%%%%%
In the meantime, from \eqref{f1-1}--\eqref{f1-2}, \eqref{f1-4} and \eqref{rep_subd070}, we can easily check that:
\begin{equation}\label{rep_subd140}
\begin{cases}
    \bigl| \frac{(z, Du)}{|Du|} \bigr| \leq \| z \|_{\infty} \leq 1, \mbox{ $ |Du| $-a.e. in $ \Omega $,} 
\\[2ex]
\quad \mbox{with } \dps \left| \int_\Omega z \cdot \nabla u \, dx \right| \leq \| z \|_{\infty} \int_\Omega |\nabla u| \, dx \leq  \int_\Omega |\nabla u| \, dx,
\\[2ex]
\quad \mbox{and } \dps \left| \int_\Omega \ryota{(z, D u)^s}  \right| \leq  \int_\Omega {\textstyle \bigl| \frac{(z, Du)}{|Du|} \bigr|} \, \ryota{|D u^s|} \leq  \int_\Omega \ryota{|D u^s|},
\end{cases}
\end{equation}
and
\begin{equation}\label{rep_subd145}
|-[z \cdot \nu]| \leq  \| z \|_{\infty} \leq 1, \mbox{ a.e. on $ \partial \Omega $.}
\end{equation}
As a consequence from \eqref{rep_subd130}--\eqref{rep_subd145}, it is inferred that: %\marginpar{$\textcolor{blue}{\frac1\tau}$}

\begin{equation}\label{rep_subd150}
\begin{cases}
\frac{(z, D u)}{|D u|} = 1, \mbox{ $ |D u| $-a.e. in $ \Omega $, and in particular,}
\\[0ex]
\hspace{5.75ex} z \cdot \nabla u = |\nabla u|, \mbox{ and } z \in \Sgn^N(\nabla u), \mbox{ a.e. in $ \Omega $,}
\\[1ex]
- %\textcolor{blue}{\frac1\tau}
[z \cdot \nu](\gamma u -v) = |\gamma u -v|, 
\\[0ex]
\hspace{5.75ex}\mbox{ i.e. } -%\textcolor{blue}{\frac1\tau}
[z\cdot \nu] \in \Sgn(\gamma u -v), \mbox{ a.e. on $ \Gamma $.}
\end{cases}
\end{equation}
%\ \ \\
Taking into account \eqref{Op.A}, \eqref{rep_subd090}--\eqref{rep_subd100} and \eqref{rep_subd150}, we infer that:
\begin{equation*}
\begin{array}{c}
(f -u, g-v) \in \mathcal{A}(u, v), \mbox{ in $ H $,} 
\\[1ex]
\mbox{i.e. } (\mathcal{A} +\mathcal{I}_H)(u, v) \ni (f, g) \mbox{ in $ H $, with } (u, v) \in D(\mathcal{A}).
\end{array}
\end{equation*}
Indeed, %we verify the 
Claim $\sharp 2$ follows.

\bigskip

Now, the rephrased conclusion \eqref{rep.subd010} will be obtained by applying Minty's theorem to $ \mathcal{A} $, and by using the maximality of the monotone graph $ \mathcal{A} \subset \partial E $ in $ H \times H $.
\end{proof}

We have just characterized $\partial E$, the subdifferential of $E$ with respect to the standard inner product of $H$. This corresponds to eq.
%Once we noted that 
(\ref{rn1}) with $\tau=1$. We would like to establish the relationship between $\partial E$  and $\partial_\tau E$, i.e. the subdifferential of $E$ with respect to the inner product $(\cdot, \cdot)_\tau$ in $H$. Thus, we could study (\ref{rn1}) for any positive  $\tau$. Here is our observation.
%This way, we will have gradient flow of functional $E(U)$ we would like to establish a similar fact for equation XXX. A result below show the relatioship between $\partial E$ and $\partial_\tau E$.

%\begin{lemma}
% An element $(\xi,\zeta)\in H$ is in the subdifferential  $\partial E(U)$ at $U$ if and only if
%  $(\xi,\frac\zeta\tau)$ is in  $\partial_\tau E$, the subdifferential of $E$ with respect to the scalar product $(\cdot,\cdot)_\tau$ in $H$.
%\end{lemma}
%\begin{proof}
%An element $(\xi,\zeta)\in H$ is a subdifferential of $E$ at $U$ if for all $h=(h_1,h_2)\in H$ we have,
%$$
%E(U+h)-E(U)\ge (\xi,h_1)_\Omega + (\zeta,h_2)_\Gamma 
%=(\xi,h_1)_\Omega + \tau (\zeta/\tau, h_2)_\Gamma
%$$
%where $(\cdot ,\cdot )_\Omega$, (resp.  $(\cdot ,\cdot )_\Gamma$), is the inner product in $L^2(\Omega)$  (resp.  $L^2(\Gamma)$).
%\end{proof}
\begin{cor}
    Let $ \tau > 0 $, and $ U = (u, v) \in H $. Then, the domain $ D(\partial E) = D(\mathcal{A}) $ coincides with the domain $ D(\partial_\tau E) $ of the subdifferential $ \partial_\tau E $ of $E$ with respect to the scalar product $(\cdot,\cdot)_\tau$ in $H$, and $(\xi,\zeta) \in \partial_\tau E(U) $ if and only if $ (\xi, \tau \zeta) \in \partial E(U) $ in $ H $.
\end{cor}
\noindent{\it Proof.}
We easily verify this lemma by using the following relationship
$$
    \bigl((\xi, \zeta), (h_1, h_2) \bigr)_\tau = \bigl( (\xi, \tau \zeta), (h_1, h_2) \bigr) \quad \mbox{for all $ (\xi, \zeta), (h_1, h_2) \in H $.}\eqno\qed
$$

%\bigskip

%\bigskip
%\hline

%A (sub)section on calibrability. Here, in this paper we talk about coherence.

\subsection{The canonical section}
%In this subsection we will set up the minimization problem determining the canonical section of the subdifferential. 
Once we described the subdifferential, we may set up the minimization necessary to select the canonical section of $\d E(U)$. Here, we assume $\tau =1$, but we shall see later that this does not lead to any loss of generality, see Lemma \ref{le6.1}.

\begin{prop}\label{s5-pr-cannon}
    If $(\xi,\zeta)$ is the canonical selection of $\d E(U)$ with $ U = (u, v) \in D(\partial E) $, then:\\
(a) $(\xi,\zeta)=(-\di z, [z\cdot \nu])$, where $z$ is a  minimizer of 
        \begin{equation}\label{c1}
    \min\left\{ \begin{array}{l|l} 
        \cE(z) & 
            \parbox{6cm}{$ z\in X_2(\Omega) $, $ z \in \Sgn^N(\nabla u) $, a.e., \\ $ \frac{(z,Du)}{|Du|} =1 $, $ |Du| $-a.e., ~ and \\ $ -[z\cdot\nu] \in \Sgn(\gamma u - v) $, $ \mathcal{H}^{N -1} $-a.e.
            }
    \end{array} \right\}, 
        \end{equation}
% \begin{equation}\label{c1}
% \min\{ \cE(z)|\ z\in X_2(\Omega), \  
% z \in \Sgn^N(\nabla u),\ \frac{(z,Du)}{|Du|} =1, \ |Du|\hbox{-}a.e.,\ 
% -[z\cdot\nu] \in \Sgn(\gamma u - v)
% \}, 
% \end{equation}
where $X_2(\Omega)$ is defined in (\ref{def-Xp}) and
$$
\cE(z)=\int_\Omega |\di z |^2 \,dx + \int_\Gamma |[z\cdot\nu]|^2\,d\cH^{N-1}.
$$
Moreover, $\di z$ and $[z \cdot \nu]$ are determined uniquely.\\
(b) We assume that $z$ is a minimizer of (\ref{c1}), $ F_0 := \{ x\in \Omega: |z(x)| <1\} $ is open, we set  $F:=\bar F_0$. If the boundaries of $ F_0 $ and $F$ are equal and they are %closure has the same 
Lipschitz continuous, %boundary, i.e. $ \d F_0 = \d F $, 
then 
$\di z = \lambda = \hbox{const}$ on $F$. \\%\todo{I think the set $ F_0 $ should be open for the variational observation, and we should exclude the case such that $ F_0 = (-2, 2)^{N} \setminus ((-1, 1)^{N -1} \times \{0\}) $.}\\
(c) If in addition to (b), we know that $|[z\cdot\nu]|< 1$ on $\Gamma_F:= F\cap \Gamma$, then
    $ \di z  = - [z\cdot\nu] ~(= \lambda)~$ on $\Gamma_F$.
\end{prop}
\begin{remark}
In particular part (c) does not apply if $|[z\cdot\nu]|=1$ on $\Gamma_F$. Parts (b) and (c) provide a set of necessary conditions for $z$ to be a minimizer. Later, in Section \ref{42CC}, we will study this in greater detail as well as we will address the sufficient conditions, see Proposition \ref{Pcon}.
\end{remark}

\begin{proof}
Part (a) follows from Theorem \ref{Prop.rep_subd} and the definition of the canonical section. Uniqueness of  $\di z$ follows from strict convexity of the integrand.

In order to establish (b) and (c), we take any smooth vector field $\varpi$, having a support in the open set $ F_0 $, such that 
$ z+t \varpi \in \Sgn^N(\nabla u) $, i.e. $ |z +t \varpi| \leq 1 $, on $ \Omega $, for all $ t \in \mathbb{R} $ with sufficiently small $|t|$. Since $ -[(z +t\varpi) \cdot \nu] = -[z \cdot \nu] $ on $ \d\Omega $ for any $ t \in \mathbb{R} $ and $z$ is a minimizer of the above problem, the function $t \in \mathbb{R} \mapsto E(z+t\varpi) \in [0, \infty) $ has a critical point at $t=0$. On the other hand it is easy to compute 
$\frac d{dt} E(z+t \varpi) \bigl|_{t=0}$. Thus, we obtain,
\begin{equation}\label{rn-div}
    \int_\Omega \di z \, \di \varpi \,dx + \int_{\partial \Omega} [z\cdot\nu] [\varpi \cdot\nu] \, d \mathcal{H}^{N -1}=0.
\end{equation}
Now, we will complete (b). We notice that
(\ref{rn-div}) simplifies if vector field $v$ has a compact support contained in the interior of $F$. In this case, the boundary term drops out, so (\ref{rn-div}) takes the form,
    $$
    \int_{F_0} \di z \,\di \varpi \,dx =0.
    $$
The integration by parts yields
$\nabla \di z = 0$ in $F_0$. Additionally, since the boundary of $ F_0 $ is Lipschitz, we can say that $ \di z = \hbox{const} =: \lambda $ on $ F$ $ (= \bar{F}_0)$.

In order to deduce (c), we take a vector field $\varpi$ having the support in 
$ F_0 \cup \{x\in \d\Omega: \ |[z\cdot\nu]|< 1\}$, which is contained in $ F_0 \cup \Gamma_F$. Then, (\ref{rn-div}) takes the form,
\begin{eqnarray*}
0&=& \int_\Omega \di z \cdot \di \varpi \,dx + \int_{\partial \Omega} [z\cdot\nu] [\varpi\cdot\nu] \, d \cH^{N -1} 
\\
&=& \int_F \lambda \, \di \varpi \,dx + \int_{\Gamma_F} [z\cdot\nu] [\varpi\cdot\nu] \, d \cH^{N -1}
= \int_{\Gamma_F} (\lambda + [z\cdot\nu]) [\varpi\cdot\nu] \,d\cH^{N-1}.
\end{eqnarray*}
Since $[\varpi\cdot\nu]$ is arbitrary, we deduce that $\lambda + [z\cdot\nu] \equiv 0$ on $\Gamma_F$. Our claims follow.
\end{proof}
%{prop} Let us suppose that $F\subset \Omega$ has Lipschitz continuous boundary.
% If a vector field $z$ is a section of a subdifferential and it is
% such that $\di z = \lambda = const$ on $F=\{x\in\Omega: |z(x)| <1\}$ and $\di z = - [z\cdot\nu]$ on $\bar F\cap \Gamma$, then $z$ is a minimizer of (\ref{c1}).
%\end{prop}
%\begin{proof} The argument is based on direct computations like that performed during the course of proof of the Proposition above.
%\end{proof}

\begin{prop}\label{s5.2-pro5.3}
Let us suppose $\Omega\subset \bR^2$ and $\Omega$, $\Gamma$  have radial symmetry, $u_0$ and $v_0$ depend only on the radius $r$. Then, for all $t>0$, if $(-\di z, [z\cdot\nu])$ is the canonical section of $\partial E(U(t))$, then we can choose %there exists 
    $z$ of the form $z(x,t) = \varrho(r,t) \frac x r$,
where $r=|x|$. % provided that $\{x: \ |z(x)|<1\} \neq \Omega$.
\end{prop}
\begin{proof} First, we notice that we can choose $z$ depending only on $r$. The argument is based on the averaging with respect to the Haar measure on $S^1$. The details are explained in \cite{gr-ade10}.

Thus, $z(x) = \varrho(r) \bne_r + \psi(r) \bne_\varphi$, where $\bne_r =\frac x r$ and $\bne_\varphi = (-x_2, x_1)/r$. We notice that $1\ge |z|^2 = \varrho^2 + \psi^2 \ge \varrho^2$. Moreover, 
$$
\di z = \di \varrho(r)\bne_r = \frac{(\varrho(r)r)'} r.
$$
In other words, the  part of $z$, tangential to circles $\d B(0,r)$, is divergence-free.

%We shall prove the second part, for this purpose 
%Let us set $F=\{x: \ |z(x)|<1\}$. Let us notice that if $x\in\Omega\setminus F$, then $Du/|Du|$ depends only on $r$. Hence  $Du/|Du|= \chi \bne_r$, where $|\chi| =1$. If $x\in \partial F$, then $1 = |z| \ge |v|$. As a result $\psi(x)=0$.
Thus, we may drop the tangential part of $z$, because it neither contributes to $\di z$, nor to the boundary trace.
%So 
%$$
%\lim_{y\to x} \bne_r \cdot z(r) 
%$$
%exists. In particlar the left and right limits are equal. Moreover, this limit is euqal to 1, because $|z(x)|=1$ for $x\in \Omega\setminus F$. Then, 
%any tangential component for $x\in F$, then 
%$1=|z(x)|^2 = |\bne_r \cdot z|^2 + |\bne_\varphi \cdot z|^2 = |\bne_r \cdot z|^2$. So, $\bne_\varphi \cdot z| =0$. As a result $z = w(r) \bne_r$.
\end{proof}
%Proof is to be copied, with changes from 
%gr-ade10.pdf
%our paper in Advances in Differential Equations
% Volume 10, Number 6, Pages 601--634
%STABILITY OF FACETS OF SELF-SIMILAR MOTION OF A CRYSTAL
\begin{remark}
 Proposition \ref{s5.2-pro5.3} extends to radially symmetric domains in $\bR^N$ and the data with the same symmetry. Here $\di z = (\varrho(r)r^{N-1})'/r^{N-1}$ for general $N$.
\end{remark}

Finally, we can decide the form of the subdifferential in the one-dimensional case, but we restrict our attention to monotone initial condition $u_0$. We set $\chi= 1$, if $u_0$ is increasing and $\chi= -1$, if $u_0$ is decreasing. We notice that the outer normals $\nu$ to $\Gamma$ are in fact numbers, $\nu(0)= -1$ and $\nu(L) =1$.
\begin{prop}\label{s5prop5.5}
    Let us suppose that $U=(u,v)\in L^2(0,T; H)$ is a solution to (\ref{rn1}), where $\Omega = (0,L)$, $\Gamma = \partial\Omega$, $u_0$ is monotone. We denote by $(\xi,\zeta)$ the canonical selection of the subdifferential of $E$ at $U(t)$, $t>0$, i.e. $\xi = - z_x$, $\zeta (i)= z(i)\cdot \nu(i)$, where $i\in \Gamma =\{0,L\}$.
 Let us consider $[a,b]\subset (0,L)$. We assume that $u_x(a^+)$, $u_x(b^-)$ exist  and they are different from zero. Then,\\
 (a) $\Sgn ( u_x(a^+) )= \Sgn ( u_x(b^-) ) = \Sgn ( u_x(x^\pm) ) =\chi$ for all $x\in [a,b]$; \\
(b) if $\gamma u = v$ at $x\in \Gamma$, then $z(i)\nu(i) \in [-1,1]$, $i\in\Gamma$;\\
  (c) if $\gamma u \neq v$ at $x\in \Gamma$, then $|z(i)|=1$, $i\in\Gamma$.
  \end{prop}
\begin{proof}
Part (a) follows from the fact that at any point $x$, where $u_0$ is differentiable and different from zero, we have $z(x) = \chi$. The set of such points in $[a,b]$ has a full measure. Since $z(x) = \chi= \Sgn(\frac d{dx} u_0(x))$ and $z$ must be continuous, we deduce that $z = \chi$ on $[a,b]$.

The proofs of the remaining parts is done by inspection of the conditions on the canonical section. 
\end{proof}

We know that the canonical selection is uniquely defined as the element of the subdifferential with the least norm. The structure of this minimization problem (\ref{c1}) is such that $z$ has to be decided only where $D u =0$. We would like to take advantage of this fact for the purpose of the %set, where $Du$ vanishes may be disconnected, we would like to 
localization of the problem. We explain it below.

\begin{cor}\label{s5c1}
Let us suppose that $(-\di z, [z\cdot \nu])$ is a canonical selection of $\d E$, 
%$F$ is the closure of connected component of the interior of $\{ x\in \Omega: |z(x)|< 1\}$ having a Lipschitz boundary 
%and $ F_0 := \bigl\{ x \in \Omega \, \bigl| \, |z(x)| < 1 \bigr\} $ is an open set such that $ \partial F_0 $ 
and $ F_0 := \bigl\{ x \in \Omega \, \bigl| \, |z(x)| < 1 \bigr\}\subset  $ is open
with Lipschitz continuous boundary %is Lipschitz 
and $ \partial F_0 = \partial \bar{F}_0 $. 
We recall that $F_0$ is contained in the complement of the support of measure $|Du|$. Let $ F $ be the closure of a connected component of $F_0 $ and let $ \nu_F $ be the outer unit normal of $ \partial F $. Additionally, let us suppose that $|[z\cdot \nu_F]|=1$ for $\cH^{N-1}$-a.e. $x\in \d F\cap \Omega$. %\todo{\PR{(1) I prefer to write `the closure of a component', because open components may have common parts of their boundaries.}}
Let $ F $ be a connected component of $ \bar{F}_0 $, and let $ \nu_F $ be the outer unit normal of $ \partial F $. Additionally, let us suppose that $|[z\cdot \nu_F]|=1$ for $\cH^{N-1}$-a.e. $x\in \d F\cap \Omega$. %\todo{\PR{ I prefer to define $F$ in terms of $z$, because the set $\nabla u=0$ is not handy,  PR}} 
%\todo{\KSh{I leave the decision on (14) of Prof. Giga's comment, up to Piotr. KS}} 

Then, $(\di z|_F, [z\cdot \nu]|_{\Gamma_F})$ minimizes the following functional
$$
I(\zeta) =\int_F |\di \zeta |^2\, dx + \int_{\Gamma_F}| [\zeta\cdot\nu]|^2\, d\cH^{N-1}
$$
in the set
$$
    \left\{ \begin{array}{l|l}
        \zeta \in L^\infty(F,\bR^N) & \parbox{6.6cm}{
            $  \| \zeta \|_{\infty} \le 1 $, $ [z\cdot \nu_F]=[\zeta\cdot \nu_F] $ on $ \partial F \cap \Omega $, $ [\zeta\cdot \nu] =0 $  on $ \d \Omega \setminus \Gamma_F $
        }
    \end{array} \right\}.
$$
\end{cor}
\begin{proof}
    Indeed, if there is $z_0$, such that $I(z_0)< I(z)$, then due to $ [z\cdot \nu_F]=[z_0\cdot \nu_F]$ on $ \partial F \cap \Omega $, we see that for $\tilde z = z_0 \chi_F + z\chi_{\Omega\setminus F}$, we have that
 $$
 \di \tilde z = \di z_0 \chi_F + \di z\chi_{\Omega\setminus F}.
 $$
 Hence, $\cE (\tilde z) < \cE(z)$, contrary to the fact that  $(-\di z, [z\cdot \nu])$ is a canonical selection of $\d E$. 
\end{proof}
In next section, we will have a closer look at $I$.

\subsection{Scaling out parameter $\tau$}
When we were calculating the subdifferential we used the standard
inner product of $H$. This corresponds to parameter $\tau =1$ in (\ref{rn1}).
We shall see here that in fact the parameter $\tau$ may be set to one
by a proper dilating of the domain $\Omega \times (0,T)$. Indeed, we can show the following statement.

%\begin{thm}
    Let us suppose $U =(u,v)\in W^{1,\infty}_{loc}([0,\infty);H) $ is a solution to (\ref{rn1}), with
  initial condition $U_0=(u_0,v_0)\in D(\partial E)$. In other words, there
    is $z\in L^\infty(0,\infty; X_2)$ and $[z\cdot \nu] \in
    L^\infty((0,\infty) \times \Gamma) $ such that
\begin{equation}\label{lsta}
 \begin{array}{ll}
   u_t = \di z & (x,t)\in \Omega\times(0,T), \\
   v_t = - [z\cdot \nu] & (x,t)\in \Gamma\times(0,T), \\
   u(x,0) = u_0(x) & x\in \Omega, \\
   v(x,0) = v_0(x) & x\in\Gamma.
  \end{array}
\end{equation}
Here, $(-\di z,[z\cdot \nu] )$ is the minimal section of
$\partial E(U)$, i.e. it is a minimizer of (\ref{c1}).
%$$
%I(u) =\int_\Omega |\di z|^2 \mathrm{d}x +  \int_{\Gamma} |[z\cdot\nu]|^2 \mathrm{d}\mathcal{H}^{N-1}.
%$$

For any $ k \in \mathbb{N} $, any $A\subset \bR^k$ and any $\tau>0$, we set,
\begin{equation}\label{dfAtau}
A_\tau = \{ \tau x: x\in A\}.
\end{equation}
Besides, we define $U^\tau=( u^\tau, v^\tau)(y,s) $, $z^\tau(y,s) $ and $ \nu^\tau(y, s) $ by the formulas
\begin{equation}\label{lstau}
    U^\tau(y,s)= U(x,t), \qquad z^\tau(y,s)=z(x,t), \quad \mbox{and} \quad \nu^\tau(y, s) = \nu(x, t),
\end{equation}
where $y =\tau x \in \Omega_\tau $, $s = \tau t \in (0, \tau T)$. We  immediately notice that
$$
U_t = \tau U^\tau_s,\qquad \hbox{div}_x \, z = \tau \hbox{div}_y \, z^\tau.
$$
If we set $\cE_\tau(\zeta)$  by formula
$$
\cE_\tau( \zeta) = \int_{\Omega_\tau} |\hbox{div}_y \, \zeta|^2 \,dy+
\frac1\tau\int_{\Gamma_\tau} |[\zeta\cdot\nu^\tau]|^2\, d\cH^{N-1},
$$
for $\zeta\in X_2$ satisfying
the conditions presented in (\ref{c1}),
then we may check (this is the content of Lemma \ref{le6.1}) that
$\cE_\tau(z^\tau) =\tau^{N-2}\cE(z)$. Thus, $z^\tau$ is the minimal
section of $\d_\tau E$. Thus, we conclude that
$U^\tau$ and $z^\tau$ form a solution to
\begin{equation}\label{lstat}
  \begin{array}{ll}
      u^\tau_t = \di z^\tau & (y,s)\in \Omega_\tau\times(0,\tau T), \\
      \tau  v^\tau_t = - [z^\tau\cdot \nu^\tau] & (y,s)\in \Gamma_\tau\times(0,\tau T), \\
      u^\tau(y,0) = u_0^\tau(y) & y \in \Omega_\tau, \\
      v^\tau(y,0) = v_0^\tau(y) & y\in\Gamma_\tau.
  \end{array}
 \end{equation}
%    Here, for any $ k \in \mathbb{N} $, any $A\subset \bR^k$ and any $\tau>0$, we set,
%\begin{equation}\label{dfAtau}
%A_\tau = \{ \tau x: x\in A\}.
%\end{equation}

In other words, we have shown:
\begin{cor}\label{cor5.2}
    If $U$ is a solution to (\ref{rn1}) in $ H = L^2(\Omega) \times L^2(\Gamma) $  with $\tau =1$, then $U^\tau$ is a
    solution to (\ref{rn1})  in $ L^2(\Omega_\tau) \times L^2(\Gamma_\tau) $ with $\tau>0$. \qed
\end{cor}

Of course, the converse statement is true.
If $\tilde U$ is a solution to (\ref{lstat}), then  $\tilde U^{1/\tau}$ is a solution to (\ref{lsta}). In order to see this, we
apply the results we have shown to $\tilde U^{1/\tau}$ and we scale $\tilde U^{1/\tau}$ by $\tau^{-1}$.
%\end{proof}

\section{Calibrability and coherence}\label{42CC} % Section 1

We shall introduce the notions of calibrability and coherency, when a facet touches the boundary of a given domain. In the following considerations, we assume that $\Omega\subset \mathbb{R}^N$ is an open bounded domain with Lipschitz boundary. 
Let $\Gamma$ be a relatively closed set in $\partial\Omega$ of positive $\mathcal{H}^{N-1}$ measure.

A compact set $F$ in $\bar{\Omega}$ together with direction $\chi \in C^0 \left(\Omega \backslash F; \{\pm 1\}\right)$ is called a \textit{facet} in $\Omega$. %\KSh{(cf. ).}\todo{\KSh{I added the refernce [Giga,Pozar,CPAM], as in (15) of Prof. Giga's comment. I leave the detailed adjustment up to Piotr. KS}}

\begin{remark}The definition above is as in \cite{GigaPozarCPAM}, however, we can also define a facet as a flat part of the graph a solution $u$, see \cite[\S 2.4]{ggr2015}. In the present context, such a distinction does not matter, because we are talking about sets, where $\nabla u =0$.
\end{remark}

Let us consider a facet $(F,\chi)$ whose boundary $\partial F$ is Lipschitz.
Let $\nu_F$ be the outer unit normal field of $ \partial F $.
 Let $z$ be a vector field in $F$ belonging to $X_2(F)$.
We say that $z$ is a \textit{Cahn--Hoffman vector field} in $F$ with $(\Omega,\Gamma)$ if
    \begin{equation}\label{CHvec}
    \|z\|_\infty \leq 1 \quad 
    %\mbox{on}\quad \Omega, \quad 
    [z\cdot\nu_F] = \gamma\chi \quad\text{on}\quad \partial F \cap \Omega, \quad
    [z\cdot\nu_F] = 0 \quad\text{on}\quad (F \cap \d \Omega) \setminus \Gamma
    \end{equation}
is %\todo{I reinstated $L^\infty$, it is more precise than $|z|\le1$} 
fulfilled, where $\gamma\chi$ is the trace of $\chi$ taken from $F^c$, the complement of $F$.
 The totality of Cahn--Hoffman vector fields is denoted by $CH(F,\Omega,\Gamma)$, i.e.,
\[
	CH(F,\Omega,\Gamma) 
    = \left\{ \begin{array}{l|l}
        z \in X_2(F) & \mbox{$ z $ fulfills \eqref{CHvec}} 
        %z \|_\infty\le 1,\ [z\cdot\nu] = \gamma\chi \ \text{on}\  \partial F\cap\Omega, \ [z\cdot\nu] = 0 \ \text{on}\  (F \cap \partial\Omega)\backslash\Gamma 
    \end{array} \right\}.
\]
We say that a facet $(F,\chi)$ with Lipschitz boundary is \textit{admissible} if $CH(F,\Omega,\Gamma)$ is non empty.

We are interested in those Cahn--Hoffman vector fields, which minimize the localized
problem of the canonical selection of $\d E$. We argued in Corollary
\ref{s5c1} that for $\tau=1$ in eq. (\ref{rn1}), the functional to minimize  was $I$, defined there. We
claim that for  
a given facet $(F,\chi)$ and a  parameter $\tau>0$  in eq. (\ref{rn1}), we must consider
the following functional in order to determine the minimal section of $E$,
\[
I_\tau(z) = \int_F |\di z|^2 \,dx 
+ \frac{1}{\tau} \int_{\Gamma_F} |[z\cdot\nu]|^2\,
d\mathcal{H}^{N-1},
\qquad z \in CH(F,\Omega,\Gamma),
\]
where $\Gamma_F:=\partial F\cap\Gamma$. We notice that $I= I_1$. Lemma
\ref{le6.1} shows the relationship between $I_1$ and $I_\tau$, proving
that  $I_\tau$ is indeed the localized
problem of the canonical selection of $\d E$.

    In principle, we should check if $I_\tau$ attains its minimum.
Functional $I_\tau$ is convex on a closed, convex set $CH(F,\Omega,\Gamma)$. We claim that
it is lower semi-continuous with respect to $L^2$-weak convergence of $\di z$. For this purpose, we have to check that the boundary integral is lower
semi-continuous. We notice that if a test function $\varphi\in W^{1,2}(\Omega)$ and $\di z_n \rightharpoonup \di z$ in $L^2$, then we may assume that $z_n \rightharpoonup z$ in $L^2$, because $ \| z_n \|_{\infty} \le 1$. As a result,
\begin{equation}\label{s6e1}
 \lim_{n\to \infty} \int_{\Gamma_F} [z_n\cdot \nu] \gamma\varphi \,d\mathcal{H}^{N-1}=
\lim_{n\to \infty} \int_F \di( z_n \varphi) \,dx =\int_F \di( z \varphi) \,dx
= \int_{\Gamma_F} [z\cdot \nu] \gamma\varphi \,d\mathcal{H}^{N-1}.
\end{equation}
In order to claim that $[z_n\cdot \nu]\rightharpoonup  [z\cdot \nu]$ in $L^2(\Gamma_F, \mathcal{H}^{N-1})$, we need to show
\begin{equation*}%\label{s6e2}
 \lim_{n\to \infty} \int_{\Gamma_F} [z_n\cdot \nu] \psi\,d\mathcal{H}^{N-1}=
    \int_{\Gamma_F} [z\cdot \nu] \psi \,d\mathcal{H}^{N-1}\qquad \hbox{for all }\psi\in L^2(\Gamma_F). %L^2(\Gamma_F, \mathcal{H}^{N-1}).
\end{equation*}
However, identity (\ref{s6e1}) combined with the standard mollification argument, yields the desired result. Hence, the boundary term is weakly lower semi-continuous, as we claimed.
 
Thus, there always exists a minimizer $z_0 \in CH(F,\Omega,\Gamma)$ of $I_\tau(z)$.
Moreover, by strict convexity of $I_\tau$ with respect to 
$\di z_0$ and $[z_0\cdot\nu]$ on $\Gamma_F$ the values of $\di z_0$ and $[z_0\cdot\nu]$
are uniquely determined although there are many minimizers $z$ of $I_\tau(z)$, besides  $z_0$.

After these preparations, the following definition is justified.
\begin{defn} \label{Cal} % Definition 1.1
An admissible facet $(F,\chi)$ is \textit{calibrable} if there is a Cahn--Hoffman vector field $z$ minimizing $I_\tau$ such that $\di z$ is constant in $F$ and that $[z\cdot\nu]$
is constant on $\Gamma$.
\end{defn}

\begin{remark} \label{RCal} % Remark 1.3
%\begin{enumerate}\item[(i)] 
The above notion of calibrability agrees with the conventional one when $\Omega=\mathbb{R}^N$.
%\item[(i\!i)] As in the case $\Omega=\mathbb{R}^N$, the vector field given in Definition \ref{Cal} is indeed a minimal Cahn--Hoffman vector field.
\end{remark}

Prompted by Definition \ref{Cal}, we introduce the notation.
$$
SCH(F,\Omega,\Gamma):= \{ z\in CH(F,\Omega,\Gamma):\ \di z= \hbox{const} \hbox{ on }F,\ [z\cdot\nu]= \hbox{const}\hbox{ on }\Gamma\}.
$$
Elements of $SCH(F,\Omega,\Gamma)$ will be called special Cahn--Hoffman vector fields.

\bigskip
We would like to establish the relationship between minimizers of $I_\tau$ and $I_1$.
\begin{lemma}\label{le6.1}
 A vector field $z\in CH( F_{\tau_0}, \Omega_{\tau_0}, \Gamma_{\tau_0})$ is a minimizer of $I_{\tau_0}$ if and only if $z^\tau \in CH( F_{\tau}, \Omega_{\tau}, \Gamma_{\tau})$, where
 $ z^\tau (x) = z\left(x \frac{\tau_0}\tau \right)$ minimizes $I_\tau$. Moreover, $I_\tau(z^\tau) = \tau^{N-2} I(z)$.
\end{lemma}
\begin{proof}
We may assume $\tau_0 =1$. 
Obviously, $z\in CH( F,\Omega, \Gamma)$ if and only if $z^\tau\in CH( F_{\tau}, \Omega_{\tau}, \Gamma_{\tau})$.

Then, we perform a change of variables in $I_\tau(z^\tau)$. Namely, after setting $y = x/\tau$, we obtain,
$$
I_\tau(z^\tau) = \int_{F_\tau} |\hbox{div}_x \, z^\tau|^2\,dx + 
 \frac 1\tau \int_{(\Gamma_F)_\tau} | [z^\tau \cdot\nu]|^2\, d\cH^{N-1}=
\tau^{N-2} I(z).
$$
This multiplicative relationship $I_\tau(z^\tau) = \tau^{N-2} I(z)$ implies  validity of our claim.
\end{proof}

The relationship between values of $z\in CH( F,\Omega, \Gamma)$ on $F$ and $\Gamma_F$ is important for our considerations. Here is our basic observation.
\begin{lem} \label{PCali} % Lemma 1.4
%\begin{enumerate}
%\item[
For $z$ in $CH(F,\Omega,\Gamma)$, we denote the average of div\,$ z$ and $[z,\nu]$ by
$$
\lambda_z = \frac{1}{|F|}\int_F \di z\,dx,\qquad \mu_z = \frac{1}{\cH^{N-1}(\Gamma_F)}\int_{\Gamma_F} [z\cdot \nu]\,d\cH^{N-1}.
$$ 
Then,
\[
\lambda_z|F| = \cH^{N-1}(\partial_+ F) - \cH^{N-1}(\partial_-F) + \mu_z \cH^{N-1}(\Gamma_F),
\]
where $|F|$ denotes the Lebesgue measure of $F$. Here,
\[
    \partial_\pm F= \left\{ x \in \partial F \cap \Omega \mid \gamma\chi = \pm 1 \right\}.
\]
\end{lem}
%\begin{proof}
\noindent{\it Proof.}
 Integration by parts yields
\begin{align*}
	\lambda_z|F| &= \int_F \di z \,dx 
    = \int_{\partial_+ F} [z\cdot\nu_F]\, d\mathcal{H}^{N-1} 
    + \int_{\partial_- F} [z\cdot\nu_F] \,d\mathcal{H}^{N-1}
	+ \int_{\Gamma_F} [z\cdot\nu] \, d\mathcal{H}^{N-1} \\
	&= \cH^{N-1}(\partial_+ F) - \cH^{N-1}(\partial_- F) + \mu_z \cH^{N-1}(\Gamma_F). \qquad\qquad \qed
\end{align*}
%\end{proof}

The above Lemma helps us to introduce a notion important in our analysis. Calibrability of a facet means that it moves as an entity. However, the bulk may move at a different velocity than the boundary layer. That is why we introduce the notion of coherency.
\begin{defn}
 We shall say that facet $(F,\chi)$ is {\it $(\tau,\Gamma)$-coherent} if there is a Cahn--Hoffman vector field $z$, minimizing $I_\tau$ such that
 $$
 \tau \lambda_z + \mu_z =0,
 $$
 where $\lambda_z$ and $ \mu_z$ defined in Lemma \ref{PCali}.
\end{defn}
Now, we are going to establish the relationship between these notions and establish the sufficient conditions for facet calibrability or $(\tau,\Gamma)$-coherency. But first we state a simple fact about quadratic polynomials.
Let $a$ and $b$ be positive constants.
We consider,
\begin{equation} \label{4f1} % (1.1)
	f(\lambda,\mu) = a\lambda^2 + b\mu^2/\tau
\end{equation}
under constraint
\begin{equation} \label{4f2} % (1.2)
	\lambda a = c + b\mu, \quad (c \in \mathbb{R}).
\end{equation}
\begin{prop} \label{El} % Proposition 1.5
    Let $(\lambda,\mu) \in \mathbb{R}\times\mathbb{R}$ be the (unique) minimizer of \eqref{4f1} subject to \eqref{4f2}, if and only if  $\tau\lambda+\mu=0$.
\end{prop}
\begin{proof}
This is elementary.
 We set
\[
	g(\mu) = f \left((c+b\mu)/a,\mu\right)
\]
and differentiate to get
\[
	g'(\mu) = 2b(c+b\mu)/a + 2b\mu/\tau.
\]
Here, we know that the minimum point $ \mu_0 $ coincides with the unique solution of $ g'(\mu_0) = 0 $.
%At the minimum point $\mu_0$ we know $g'(\mu_0)=0$.
Additionally, due to the constraint $\lambda=(c+b\mu)/a$, the equation $g'(\mu_0)=0$ is equivalent to $\lambda_0+\mu_0/\tau=0$, with $\lambda_0=(c+b\mu_0)/a$. The proof of Proposition is now complete.
\end{proof}
We will use this observation in the next Proposition, providing sufficient conditions for a minimizer.

\begin{prop}\label{Pcon}
If $z_*\in CH(F,\Omega,\Gamma)$ is such that $\tau \lambda_* + \mu_* =0$, where $\lambda_* = \lambda_{z_*} = \di z_*$ on $F$, $\mu_* = \mu_{z_* }= [z_*\cdot\nu]$ on $\Gamma_F$, then $z_*$ is a minimizer of $I_\tau$.
\end{prop}
\begin{proof}
By the Schwarz inequality, for any vector field $z$, we have
\[
\frac{1}{|F|} \left(\int_F \di z \,dx\right)^2
\leq \int_F (\di z)^2 \,dx, \quad
\frac{1}{\cH^{N-1}(\Gamma_F)} \left(\int_{\Gamma_F} [z\cdot\nu] \,d\mathcal{H}^{N-1}\right)^2
\leq \int_{\Gamma_F} [z\cdot\nu]^2\, d\mathcal{H}^{N-1}.
\]
Thus, we see that the definition of $I_\tau$ yields,
\begin{equation}\label{nierw}
|F| \lambda_z^2 + \frac{\cH^{N-1}(\Gamma_F)}\tau \mu_z^2 \le 
\int_F (\di z)^2\,dx + \frac1\tau \int_{\Gamma_F} [z\cdot \nu]^2\,d \cH^{N-1} = I_\tau(z).
\end{equation}
We know by Lemma \ref{PCali} that 
\begin{equation}\label{consr}
 |F| \lambda_z = c + \mu_z \cH^{N-1}(\Gamma_F), 
\end{equation}
where $c$ is a constant. By Proposition \ref{El}, the left-hand-side of (\ref{nierw}) is minimized under the constraint (\ref{consr}) if and only if $\tau \lambda_{z_*} + \mu_{z_*} =0$. Furthermore, our assumption on $z_*$ yields us
$$
I_\tau (z_*) = |F| \lambda_*^2 + \frac{\cH^{N-1}(\Gamma_F)}\tau \mu_*^2 \le  I_\tau(z).%\eqno\qed
$$
Thus, the proof is complete.
%In both inequalities the equality holds if and only the integrand is a constant function.
\end{proof}

In other words, a special Cahn--Hoffman vector field $z$ minimizes $I_\tau$ if we can ensure $\tau \lambda_{z} + \mu_{z} =0$.
Now, we may prove the converse statement.
\begin{prop}\label{pr3prime}
 Let us suppose that facet $(F,\chi)$ is calibrable and $(\tau,\Gamma)$-coherent. Then, there is $z_*\in SCH(F,\Omega,\Gamma)$ such that $\tau \di  z_* + [z_*\cdot\nu] =0$.
\end{prop}
\begin{proof}
Since $ (F, \chi) $ is calibrable, we can find $ z_* \in SCH(F, \Omega, \Gamma) $ minimizing $ I_\tau $. In addition, we must have $\tau \di  z_* + [z_*\cdot\nu] =0$, due to the $ (\tau, \Gamma) $-coherency of $ (F, \chi) $ and the uniqueness of minimizing $\di z_*$ and $[z_*\cdot\nu]$.
%    Let us take a minimizer, $z_*\in CH(F,\Omega,\Gamma)$, of $I_\tau$. Then, by uniqueness of minimizing $\di z_*$ and $[z_*\cdot\nu]$ we must have \KSh{} $\tau \di  z_* + [z_*\cdot\nu] =0$.
\end{proof}

%We are now ready to introduce the notion of coherency.
%
%\begin{defn} \label{Coh} % Definition 1.2
%A facet $(F,\chi)$ is $(\tau,\Gamma)$-\textit{coherent} if it is calibrable and weakly coherent.
%\end{defn}

However, we must be prepared for the existence of calibrable facets, which are not $(\tau,\Gamma)$-coherent.

\begin{lem} \label{PCal}
    If for a facet $(F,\chi)$ there is $z_*$ in $SCH(F,\Omega,\Gamma)$, with $ \lambda := \di z_* \in \mathbb{R} $ and $ \mu := [z_* \cdot \nu] \in \mathbb{R} $, such that $|\mu| = 1$ and $(\lambda + \mu/\tau) \Sgn \mu \le 0$. Then, this facet is calibrable.
\end{lem}
\begin{proof}
Let us suppose that $z_* \in SCH(F,\Omega,\Gamma)$ and $\di z_* =\lambda$, $[z_*\cdot\nu] =\mu$. %, where $\lambda\in \bR$, $ |\mu|=1$. 
We take a test vector field $\zeta$ such that $z_*+ \zeta \in CH(F,\Omega,\Gamma)$. Thus, in particular, $[\zeta\cdot \nu_F] = 0$ on $\d F \setminus \Gamma$, but $[\zeta\cdot\nu]\Sgn \mu \le 0$ on $\Gamma_F$. 

After having performed simple computations and the integration by parts, we arrive at
\begin{align}\label{l6.4eq1}
    I_\tau(z_*+ \zeta) -I_\tau(z_*)- I_\tau(\zeta) ~ & =
    2\lambda\int_F \di \zeta\,dx + {\frac{2\mu}{\tau}} \int_{\Gamma_F} [\zeta\cdot\nu]\,d\cH^{N-1} 
    \nonumber
    \\
    & =
%I_\tau(z_*)+ I_\tau(\zeta) + 
2\left( \lambda + \frac\mu\tau \right)\int_{\Gamma_F} [\zeta\cdot\nu]\,d\cH^{N-1}.
\end{align}
Thus, combining this information with $(\lambda + \mu/\tau) \Sgn \mu \le 0$ and  $[\zeta\cdot\nu]\Sgn \mu \le 0$ on $\Gamma_F$
yields that the right-hand-side in (\ref{l6.4eq1}) is positive. Hence, our claim follows.
\end{proof}
\begin{thm} \label{CO} % Theorem 1.6
Let $(F,\chi)$ be calibrable. %\marginpar{check}
 Assume that \eqref{4f1} is minimized under \eqref{4f2} at some $(\lambda_0,\mu_0)$ with $|\mu_0|\leq 1$.
 Assume that there is $z_* \in SCH(F,\Omega,\Gamma)$ with $[z_*\cdot\nu]=\mu_0$.
 Then $(F,\chi)$ is $(\tau,\Gamma)$-coherent.
\end{thm}
\begin{proof}
This is just an application of Proposition \ref{El}.
 The assumption $|\mu_0|\leq 1$ is a necessary condition so that $[z_*\cdot\nu]=\mu_0$ since $\|z_*\|_\infty \leq 1$.
\end{proof} 

The property of coherency heavily depends on geometry of $\Gamma$.
 Here is a conjecture.
\begin{con}\label{e-con} % Conjecture 1.7
If $\Gamma$ is strictly mean-convex near $F$, then an admissible facet $(F,\chi)$ is $(\tau,\Gamma)$-coherent for all $\tau>0$.
Here, when we say that $\Gamma$ is strictly mean-convex, we mean that there is a positive constant $\gamma_0$ such that $\kappa \geq \gamma_0$ on $\Gamma$, where $\kappa$ is the inward mean curvature of $\Gamma$ in $\partial\Omega$.

More
generally, if $\inf_{x\in\Gamma \cap F} \kappa(x) > -1/\tau,$ then we expect that $(F,\chi)$ is $(\tau,\Gamma)$-coherent.
\end{con}
%\todo{\KSh{I leave the decision for (16) of Prof. Giga's comment up to Piotr. KS}}
%
We shall show this conjecture with extra regularity assumptions on a minimizer of $I_\tau(z)$.
\begin{thm} \label{MC} % Theorem 1.8
Assume that $\partial\Omega$ is at least $C^2$ in a neighborhood of $\Gamma$ and the mean curvature $\kappa$ is estimated from below, $\kappa \geq \gamma_0$ where %some constant 
$\gamma_0>1/2-1/\tau$.
Let $z_0 \in CH(F,\Omega,\Gamma)$ be a minimizer of $I_\tau$.
Assume that $z_0$ can be extended as a $C^2$ function in a neighborhood $U$ of $\Gamma_F$ in $\mathbb{R}^N$. Then, the following properties hold:
%\todo{\KSh{I think we would need $ C^2 $-regularity of $ z_0 $ in the proof of item (ii).}}
\begin{enumerate}
\item[(i)] $|[z_0\cdot\nu]|<1$ on $\Gamma_F$,
\item[(i\!i)] $\tau\di z_0 + [z_0\cdot\nu]=0$ on $\Gamma_F$.
 In other words, $(F,\chi)$ is $(\tau,\Gamma)$-coherent.
\end{enumerate}
\end{thm}
\begin{proof} 
Let $d_\Sigma$ be the distance function from a closed subset $\Sigma$ of $\Gamma_F$.
 We recall a general formula
\begin{equation} \label{4DF} % (1.3)
	\di z = \hbox{div}_T z + (m \cdot \nabla)(z \cdot m),
\end{equation}
where $\hbox{div}_T$ is the surface divergence on a hypersurface $\{d_\Sigma=c\}$ and $m=-\nabla d_\Sigma$, which is normal to $\{d_\Sigma=c\}$.
 Indeed,
\begin{align*}
	& \hbox{div}_T z = \mathrm{tr}(I - m \otimes m)\nabla z \\
	& \mathrm{tr}(m \otimes m\; \nabla z) = \sum_{i,j} m_i m_j \partial_j z_i
	 = (m \cdot \nabla)(z \cdot m)
\end{align*}
because $(m \cdot \nabla)m=0$.
 This implies the desired decomposition \eqref{4DF}.
 The formula \eqref{4DF} holds for $\text{a.e.}\ c$ and for $\mathcal{H}^{N-1}\text{-a.e.}\ x\in\{d_\Sigma=c\}$ if $d_\Sigma$ is not $C^1$, but Lipschitz continuous.
%
%\begin{enumerate}
%\item[

\medskip\noindent(i) We set $\Sigma=\{ x\in\Gamma_F \mid [z_0\cdot\nu]=1\}$ and set
\[
	z_\varepsilon = \varphi_\varepsilon z_0 \quad\text{with}\quad 
	\varphi_\varepsilon = \min(1,1-\varepsilon+d_\Sigma),
\]
for $\varepsilon>0$.
 We shall prove that $I_\tau(z_\varepsilon)<I_\tau(z_0)$ for sufficiently small $\varepsilon>0$ assuming that $\Sigma$ is non empty.
 We calculate
\begin{align*}
	I_\tau(z_0)-I_\tau(z_\varepsilon)
	& = \left\{ \int_{\Sigma_\varepsilon} |\di z_0|^2 \,dx
		- \int_{\Sigma_\varepsilon} |\di z_\varepsilon|^2 \,dx \right\} \\
	& + \frac{1}{\tau}\left\{ \int_{\Gamma_F} ([z_0\cdot\nu])^2\, d\mathcal{H}^{N-1}
		- \int_{\Gamma_F} (z_\varepsilon\cdot\nu)^2 \,d\,\mathcal{H}^{N-1} \right\} \\
	& =I + I\!\!I,
\end{align*}
where $\Sigma_\varepsilon = \left\{ x\in F \mid d_\Sigma(x) < \varepsilon \right\}$.
To estimate $I$, we calculate
\[
	\di z_\varepsilon = \varphi_\varepsilon \di z_0 
	+ \nabla\varphi_\varepsilon \cdot z_0 
\]
and observe that
\[
	I \geq - 2 \int_{\Sigma_\varepsilon} (\varphi_\varepsilon \di z_0) \nabla\varphi_\varepsilon \cdot z_0 \,dx 
	- \int_{\Sigma_\varepsilon} (\nabla\varphi_\varepsilon \cdot z_0)^2\, dx 
	= I\!\!I\!\!I + I\!V.
\]
We use the decomposition formula \eqref{4DF} %\textcolor{blue}
and the fact $z\in C^1(U)$ to get
\begin{align*}
	I\!\!I\!\!I/2
	& = - \int_{\Sigma_\varepsilon} (\varphi_\varepsilon \di z_0) \nabla\varphi_\varepsilon \cdot z_0\, dx 
	= \int_{\Sigma_\varepsilon} (\varphi_\varepsilon \hbox{div}_T z_0)m \cdot z_0 \, dx \\
	& + \int_{\Sigma_\varepsilon} \left(\varphi_\varepsilon(m \cdot \nabla)(z_0 \cdot m)\right)(z_0 \cdot m) \,dx \\
	& = \varepsilon \int_\Sigma (\hbox{div}_T \nu)([z_0\cdot\nu]) \,d\mathcal{H}^{N-1} + O(\varepsilon^2) \\
	& + (1-\varepsilon) \int_{\Sigma_\varepsilon} (m\cdot\nabla)\left(\frac{(z_0 \cdot m)^2}{2}\right)\, dx + O(\varepsilon^2) \quad\text{as}\quad \varepsilon\to 0.
\end{align*}
Since $[z_0\cdot\nu]=1$ on $\Sigma$, $|z_0| \leq 1$ in $F$ and  %textcolor{blue}{[I'd rather say:Since 
$z_0\in C^1(U)$, then
\[
	 \int_{\Sigma_\varepsilon} (m\cdot\nabla)(z_0 \cdot m)^2 \,dx \geq  O(\varepsilon^2).
\]
        Finally, since $\kappa = \di_T \nu$, we observe that
\[
	I\!\!I\!\!I \geq 2 \varepsilon \int_\Sigma \kappa \,d\mathcal{H}^{N-1} + O(\varepsilon^2).
\]
It is easy to see that
\[
    I\!V = - \varepsilon \int_\Sigma (z_0\cdot \nu)^2 \,d\mathcal{H}^{N-1} + O(\varepsilon^2) 
	= - \varepsilon \int_\Sigma \,d\mathcal{H}^{N-1} + O(\varepsilon^2).
\]
Thus, we observe that 
\[
	I \geq \varepsilon \left( \int_\Sigma 2\kappa\,d\mathcal{H}^{N-1} 
	- \int_\Sigma \,d\mathcal{H}^{N-1} \right) + O(\varepsilon^2).
\]
It is easy to estimate 
\begin{align*}
	\tau I\!\!I 
	& = \int_{\Gamma_F} \left((z_0 - z_\varepsilon)\cdot\nu\right)\left((z_0 + z_\varepsilon)\cdot\nu\right) \,d\mathcal{H}^{N-1} \\
 & = \varepsilon \int_\Sigma ([z_0\cdot\nu]) \cdot (2 [z_0\cdot\nu]) \,d\mathcal{H}^{N-1}  + O(\varepsilon^2) \\
	& = 2 \varepsilon \int_\Sigma\, d\mathcal{H}^{N-1} + O(\varepsilon^2).
\end{align*}
Thus,
\[
	I+I\!\!I \geq \varepsilon \left( \int_\Sigma 2\kappa\,d\mathcal{H}^{N-1}
	+\left(\frac{2}{\tau}-1\right)
	\int_\Sigma \,d\mathcal{H}^{N-1}\right) + O(\varepsilon^2).
\]
As a result, if $\kappa \geq \gamma_0$ with $\gamma_0 > 1/2-1/\tau$, then $I_\tau(z_0)-I_\tau(z_\varepsilon)>0$ for small $\varepsilon$.
 We thus prove that $z_0 \cdot \nu<1$ if $z_0$ is a minimizer.
 The inequality $z_0 \cdot \nu >-1$ can be proved similarly.

%\item[
\medskip\noindent(i\!i) If $|[z_0\cdot\nu]|<1$ on $\Gamma_F$, then for any $h\in C^1(F)$ such that $\left|(z_0+\varepsilon h)\cdot\nu\right|<1$ on $\Gamma_F$ and $|z+\varepsilon h|\leq 1$ in $F$.
 The first condition does not restrict $h_\nu=h\cdot\nu$.
 The second condition restricts the tangential component $h_T=h-(h\cdot\nu)\nu$ so that $|z_0 + \varepsilon h_\nu\cdot\nu + \varepsilon h_T|\leq 1$.
 Since $z_0$ is the minimizer, we obtain that 
\begin{align*}
	0 & = \frac{1}{2} \left.\frac{d}{d\varepsilon}\right|_{\varepsilon=0} I_\tau(z_0+\varepsilon h) 
	= \int_F \di z_0 ~ \di h\, dx 
	+ \frac 1\tau\int_{\Gamma_F}([z_0\cdot\nu])h_\nu\, d\mathcal{H}^{N-1} \\
	& = - \int_F \nabla\di z_0\cdot h\,dx
    + \int_{\partial F} \di z_0 \, (h \cdot \nu_F) \, d\mathcal{H}^{N-1}
	+ \frac 1\tau \int_{\Gamma_F}([z_0\cdot\nu])h_\nu \,d\mathcal{H}^{N-1}. 
\end{align*}
We take arbitrary $C^1$ function $f$ near $\Gamma_F$ and 
%set $h_\nu=f$ on $\Gamma_F$. We extend $h$ outside $\Gamma_F$ to $F$ so that its support is in an $\delta$-neighborhood of $\Gamma_F$ and $h$ is $C^1$. Moreover,  
pick a test function $ h \in C^1(F) $, satisfying %that 
$ h_\nu = f $ on $ \Gamma_F $. We require that the support of $ h $ is in an $ \delta $-neighborhood of $ \Gamma_F $, and moreover, the
tangential component is controlled by $f$.
 The first term of above formula is $O(\delta)$ as $\delta\to 0$, as a result
\[
	\int_{\Gamma_F} (\tau\di z_0 + [z_0\cdot\nu])f\,d\mathcal{H}^{N-1} = 0.
\]
This implies that
\[
	\tau\di z_0 + [z_0\cdot\nu] = 0 \quad\text{on}\quad \Gamma_F.
\]
%\end{enumerate}
%
\end{proof} 

Actually, we could relax the assumptions of this Theorem without weakening the claim.
\begin{cor}\label{cor6.1}
 The conclusion of Theorem \ref{MC} holds if we assume that the mean curvature $\kappa$ is estimated as follows,
 $$
 \inf_{\Gamma_F} \kappa> - \frac 1\tau.
 $$
\end{cor}
\begin{proof}
 We fix $\tau =\tau_0$, we consider $\Omega_{\tau/\tau_0}$, (see formula (\ref{dfAtau})),
 %:=\{ y = \frac{\tau}{\tau_0} x: x\in\Omega\}$ 
 a scaled domain. %\marginpar{check consistency of defs}
 We notice that $\kappa_{\tau/\tau_0}$, the mean curvature of $(\Gamma_F)_{\tau/\tau_0}$, is equal to $\frac{\tau_0}{\tau}\kappa.$ %\marginpar{check H}
 Our condition $\inf_{\Gamma_F} \kappa> \frac12 -\frac 1{\tau_0}$ is equivalent to 
 $$
    \inf_{(\Gamma_F)_{\tau/\tau_0}} \frac{\tau}{\tau_0} \kappa_{{\tau}/{\tau_0}} > \frac12 -\frac 1{\tau_0}.
 $$
In other words, $\kappa_{\tau/\tau_0} > \frac{\tau_0}{2\tau} - \frac 1\tau.$  However, we may take arbitrary small $\tau_0$. This means that $\kappa$ may be as close to $-\frac 1\tau$, as we wish. Thus, the condition
$$
\inf_{\Gamma_F} \kappa > -\frac1 \tau
$$
is sufficient to guarantee existence of a Cahn--Hoffman vector field.

Suppose now that
$$
\inf_{\Gamma_F} \kappa > -\frac1 \tau,
$$
then there exists $\tau_0>0$ such that $\inf_{\Gamma_F} \kappa > \frac{\tau_0}{2\tau}-\frac1 \tau$. Now, by scaling we consider the problem in $\Omega_{\tau_0/\tau}$, then the mean curvature condition is equivalent to 
$$
\inf_{\Gamma_F} \frac{\tau_0}{\tau} \kappa_{\tau_0/\tau} >  \frac{\tau_0}{2\tau}-\frac1 \tau,
$$
because $\kappa = \frac{\tau_0}{\tau} \kappa_{\tau_0/\tau}$. In other words,
$$
\kappa_{\tau_0/\tau} >  \frac 12-\frac1 {\tau_0},
$$
as desired.
\end{proof}

\section{Instant facet formation in the one-dimensional problem}\label{sone}
Before considering any two dimensional configuration, we would like to present a simple one-dimensional warm-up problem. We assume that our data contain exactly one facet touching the boundary, where we specify the dynamic boundary condition. Our goal is to capture the behavior of facets by constructing explicit solutions. This task involves monitoring the boundary behavior of solutions.

For the sake of
simplicity, we consider only monotone initial condition $u_0$. We shall write 
\begin{equation}\label{def-chi}
\chi = 1\quad\hbox{if}\quad u_0\quad \hbox{is increasing and}\quad
\chi = -1\quad\hbox{if}\quad u_0\quad \hbox{is decreasing.}
\end{equation}

%$[$NOTE: NEED TO DEFINE THE NOTION OF A `FACET'. FOR PR A FACET IS A SPECIAL PART OF THE GRAPH OF A SOLUTION.$]$

\begin{thm}\label{5thm1}
Let us suppose that $\Omega = (0,L)$, $\Gamma = \d\Omega$ and $u_0\in \Lip(\Omega)$. We assume that $u_0$ 
is strictly monotone on $[0,b_0]$, $b_0\in(0,L)$ and it has exactly one facet $([b_0,L],\chi)$  touching the boundary at $x=L$  and $v_0= u_0|_\Gamma$. 
 Then,\\
 1) %If the left endpoint of the facet is $b$, then 
 The vertical velocity of facet  $([b(t),L],\chi)$ is
\begin{equation}\label{staff}
 u_t = \frac{-\chi}{1+L-b(t)}.
\end{equation}
Facet $([b_0,L],\chi)$ calibrable and $(1,\Gamma)$-coherent. 
Moreover, if we write $h_r(t) = u(\cdot, t)|_{[b(t),L]}$, then $\frac {d h_r}{dt} = u_t$ and $h_r(0) = u(L,0)$. In addition, $b(t)$ is a solution to $h_r(t) = u_0(b(t))$, i.e. $b(t)= (u_0)^{-1}(h_r(t))$.\\
2) A new facet forms instantly at $x=0$, its initial velocity is $\chi$. If we denote by $a(t)$ the right endpoint of the new facet at $t>0$, then its velocity is given by
\begin{equation}\label{staf2}
 u_t  = \frac{\chi}{1+a(t)}.
\end{equation}
Facet $([0,a(t)],\chi)$ calibrable and $(1,\Gamma)$-coherent. 
Moreover, if we write $h_l(t) = u(\cdot, t)|_{[0,a(t)]}$, then $\frac{d h_l}{dt} = u_t$ and $h_l(0) = u(0,0)$. In addition, $a(t) = (u_0)^{-1}(h_l(t))$.
In particular, the velocity of facet $[0,a(t)],\chi)$ is continuous at $t=0$.\\
3) The unique solution $U(t)=(u(t), v(t))$ is given by formula (\ref{5r-sum}) below. In particular,
for all $t\ge 0$, $\gamma u(t) = v(t)$, in other words, $\gamma u_t = v_t$.
\end{thm}

\begin{remark} Roughly speaking, the sign of the velocity of facet $([b(t),L],\chi)$ is opposite to the sign of the space derivative of solution $u$ on $(a(t),b(t))$, while in case of facet $([0,a(t)],\chi)$ the signs of its velocity and $ u_x$  on $(a(t),b(t))$ agree.
% 1)  The facet endpoints $a$ and  $b$ depends on time. It is possible to work out details by using the method described in \cite{MR3020135}.\\ %\cite{kiemury}.
% 2) If $a(\cdot)$ 

We notice that %if 
$a$ is a continuous function of time and so is the velocity of the facet at $x=0$.
\end{remark}
{\it Proof of Theorem \ref{5thm1}.} %The argument is presented for the case of increasing $u_0$. For decreasing $u_0$ a factor $\chi$ is necessary.
We are going to construct semi-explicit solutions from the information about the subdifferential $\d E(U)$. We will use the fact that the solution $U:[0,+\infty) \to H$ with the initial condition $ U(0) = U_0 \in D(\partial E) $ is a locally Lipschitz continuous function. Moreover, for all $t\ge0$, we have $U(t) \in D(\d E)$ and 
\begin{equation}\label{dd}
 \frac{d U^+}{dt} = - \mathcal{A}^o(U), \qquad\hbox{for all }t\ge 0,
\end{equation}
where $\mathcal{A}^o(U)$ is the canonical section of  $\d E(U)$. Once we construct $\mathcal{A}^o(U_0)$, we will argue 
about the formula for the solution, which can be checked directly.
%that at later sufficiently times  the form of $\KSh{\mathcal{A}^o}(U)$ is the same.

1) We learn from Proposition \ref{s5prop5.5} that the subdifferential of $E(U_0)$ has the form $(\xi, \zeta) = (-z', z(L))$, where $z(x) \in \Sgn(\frac d{dx} u_0(x))$ and 
$z(L)\in \Sgn(u_0(L) - v_0)= \Sgn \, 0$.
Since $u_0$ is a.e. differentiable on $(0,b_0)$, then the continuous $z$ must be equal to
$\Sgn (\frac d{dx} u_0(x))$ for a.e. $x\in (0,b_0)$. Hence,  $z(x)=\chi$ for $ x\in [0,b_0]$. Moreover,
%It is easy to see that
by Proposition \ref{s5-pr-cannon},
the optimal $z$ has to be linear on facets. Hence, $z$ takes the following form,
$$
z(x)=
%\left\{
%\begin{array}{ll}
% 1, & x\in (0,b)\\
 \frac{\mu -\chi}{L-b_0} (x-b_0) + \chi,\qquad x\in (b_0,L),
%\end{array} \right. 
$$
where $z(L) =\mu\in[-1,1]$ has to be determined.

The variational problem set in (\ref{c1}) leads to the following simple question of  minimization
$$
\int_{b_0}^L \left(\frac{\mu -\chi}{L-b_0}\right)^2 \,dx + \mu^2
$$
with respect to $\mu \in [-1,1]$. The minimum is attained if and only if
\begin{equation}\label{5r-eta}
 \frac{\mu -\chi}{L-b_0} + \mu =0,
\end{equation}
i.e. 
\begin{equation}\label{5r1}
 \mu = \frac{\chi}{1+L-b_0}.
\end{equation}
We notice that due to (\ref{5r-eta}), after we identify $v_t(L)$ with $\left.\frac{d^+}{dt}v(0)\right|_{x = L}$, then we have,
\begin{equation}\label{5r-star}
%    \KSh{v_t(L) ~ \left( = \left.\frac{d^+}{dt}v(0)\right|_{x = L} \right)} 
v_t(L)= - [z\cdot\nu](L) = - \mu = \di z.
\end{equation}
In other words, facet $([b_0,L],\chi)$ and the boundary value move with the same velocity. 
According to the theory developed in Section \ref{42CC} and formula (\ref{5r-star}), we conclude that facet $([b_0,L],\chi)$ is calibrable and $(1,\Gamma)$-coherent.

Moreover, formula (\ref{5r-star}) shows that facet $([b(t),L],\chi)$ will expand, because if we set $h_r(t) = u(\cdot, t)|_{[b(t),L]}$, then $h_r(t)$ must satisfy the equation,
$$
u_0(b(t)) = h_r(t),\quad t\ge 0
$$
or
$$
b(t) = u_0^{-1}(h_r(t))
$$
which is well-defined for the expanding facet due to monotonicity of $u_0$. We expect that (\ref{5r-star}) will continue to hold for later times, which combined with the above formula for $b$ yields the following ODE for $h_r(t)$,
\begin{equation}\label{5r-hr}
\frac d{dt} h_r = \frac{-\chi}{1+L - u_0^{-1}(h_r)},\qquad h_r(0) = u_0(L).
\end{equation}
The right-hand-side of this equation need not be Lipschitz continuous, but it is a decreasing function, thus there is a unique solution to (\ref{5r-hr}).

2) Now, we turn our attention to  $x=0$. We note that at $t=0$, we have
$$
v_t (0)= - [z\cdot\nu]  = \Sgn\, u'_0(0)= \chi.
$$
This condition means that a facet forms instantaneously and it moves with initial velocity $\chi$. 
It is easy to derive a formula for  the velocity of this facet. Namely, the argument which leads us to (\ref{staff})  yields (\ref{staf2}) too. Moreover, if we set $h_l(t) = u(\cdot,t)|_{[0,a(t)]}$, then the expanding facet $([0,a(t)],\chi)$ must satisfy the equation $u_0(a(t)) = h_l(t)$. Thus, we come to the conclusion that $h_l$ must satisfy the following ODE 
\begin{equation}\label{5r-hl}
\frac d{dt} h_l = \frac{-\chi}{1+ u_0^{-1}(h_l)},\qquad h_l(0) = u_0(0).
\end{equation}
Moreover, facet $([0,a(t)],\chi)$ is calibrable and $(1,\Gamma)$-coherent, (for $t>0$). We may use the same argument, as in part 1), to establish this result. 

3) In the above considerations, facets attached to $\Gamma$ always move with the same velocity as the boundary layer. Moreover, we see that on $[a(t), b(t)]$, we have $u_t=0$. Hence, we can summarize our computations in the following formula for $u$,
\begin{equation}\label{5r-sum}
 u(x,t) =\left\{ 
 \begin{array}{ll}
  h_l(t) & x \in [0,a(t)],\\
  u_0(x)& x \in (a(t), b(t)),\\
  h_r(t)& x \in [b(t),L],
 \end{array}
\right.
\end{equation}
where $t<T_{cr}$ and $a(T_{cr}) = b(T_{cr})$.

We can check by inspection that $U(\cdot, t) = (u(\cdot,t), \gamma u(\cdot,t))$ belongs to $D(\partial E)$ and it satisfies (\ref{dd}). Our claim follows. \qed
%the  Indeed, (\ref{5r-star}) and its analogue for $x=0$ tell us that the velocity of $v(i)$, $i\in \Gamma$, equals the velocity of the facet touching $i$. 

%We notice that in all cases above the speed $|u_t|$ is bounded and always $|v_t|$ does not exceed 1. Moreover, since the facets we studied had the tendency to grow, we deduce that $t\to U(t)$ is continuous not only in the $\cH$  topology but also on $L^\infty(\Omega) \times L^\infty(\Gamma)$. Thus, the observations we made at $t=0$ are valid also for small positive $t$. \qed

\section{A boundary layer behavior %detachment 
in the radial case in two dimensions}\label{stud}
%\vspace{-3ex}
%\todo{Please comment the reason why the 2D restriction is needed. In fact, I am not sure the reason.}

We would like to study properties of solution while taking advantage of the radial symmetry. We expect that the examples, we are going to present, look the same in all dimensions bigger than one. However, for the sake of definiteness, we restrict our attention to the planar case.
%\todo{\KSh{I agree with the policy. Also, I leave the decision on (17) of Prof. Giga's comment, up to Piotr. KS}}

If $\Omega$ is a ball centered at the origin and $\Gamma = \d\Omega$, then we will see that
Corollary \ref{cor6.1} implies that 
any radially symmetric facet $(F,\chi)$ touching $\Gamma$ will be calibrable and $(1,\Gamma)$-coherent. Nonetheless, we find it instructive to present the construction of minimizers of $I_\tau$, (here $\tau =1$), based on Proposition \ref{Pcon} and Lemma \ref{PCal}.

We also consider  the case when $\Omega$ is an annulus and $\Gamma$ is the boundary of the inner ball, then Theorem \ref{CO} in general is not applicable. Thus, we may expect to see the boundary layer detachment. In other words, we will see calibrable facets, which are not coherent. We will present detailed calculations in Subsection \ref{ssann}.
%For the sake of clarity of exposition. 
%We first present computation for the dynamic boundary conditions, then we will carefully adjust them, to reflect the Neumann and Dirichlet conditions. \marginpar{to be discussed}
Our argument depends on the form of the canonical selection of the subdifferential presented in 
Proposition \ref{s5.2-pro5.3}.

We state our results for $\tau=1$, which has the obvious advantage of notation simplicity. However, Corollary \ref{cor5.2} and Lemma \ref{le6.1} tell us that once we have a result for  $\tau=1$, then we have the same result for any $\tau>0$ on a scaled domain. We leave the details for the interested reader.

\subsection{A ball}\label{ssball}
We want to take advantage of a possible simplification of the argument, when we consider a ball $B(0,R)$. 
We assume the radial symmetry of initial datum $u_0(x)= u_0(|x|)$. We restrict our attention to Lipschitz continuous $u_0$ and monotone $r\mapsto u_0(r)$ data. We use $\chi$ as defined in (\ref{def-chi}).

We assume that $\Gamma=\partial B(0,R)$ and we set $F:=\bar B(0,R)\setminus B(0,\rho)$ with $ \rho \in (0, R) $. We will consider facet $(F,\chi)$, where both $F$ and $\chi$ are defined above, which
is attached to  the boundary.  

It is clear that a facet at the center of the ball must appear, see \cite{MR2746654}, however, in order to simplify the discussion, we assume that a facet
$B(0,a_0)$, $a_0>0$ is already present,  i.e. $U_0 \in D(\d E)$.
%\todo{For readers, it would be better to show some reference(s) here, e.g. \cite{MR2746654}.}.  
We will denote its  evolving radius by 
$a\equiv a(t)$. We are not interested in any other interior facets. 

We want to discover only short time dynamics before any possible facet collision occurs at $t=t_{cr}$.
We argue, as in the proof of Theorem \ref{5thm1}, that for this purpose, it is sufficient to study $\mathcal{A}^o(U_0)$, the canonical section of $\d E(U_0)$, because $U\in C([0,+\infty);H)$, $U(t) \in D(\d E)$ for all $t\ge0$ and 
\begin{equation}\label{r81-1}
 \frac{d^+ U}{dt} + \mathcal{A}^o(U) =0\qquad \hbox{for all }t \ge 0.
\end{equation}
The explicit form of the minimal section of $\partial E$ at $U_0$ is such that the formulas for the position of the facet
depend continuously upon parameters, %change much during the evolution, 
thus we can directly check that (\ref{r81-1}) holds
%We shall see that $|U^+_t(t)|$ is bounded by $|U^+_t(0)|$, thus $U(\cdot)$ is continuous not only into $H$ but also into $L^\infty(\Omega) \times L^\infty(\Gamma)$. As a result, the conclusion we drew about $U_0$ are valid also for a short time, 
until a facet collision occurs.

Now, we begin our analysis of the subdifferential. In fact, it is sufficient to consider the localized functional $I_\tau$, in this case $\tau =1$.

In principle, we have the following cases singled out in Proposition \ref{s5-pr-cannon}. If $z$ is a minimizer of (\ref{c1}), i.e. $I_1$, then we have either:
\\[1ex]
(1) Facet $(F,\chi)$ is calibrable and $(1,\Gamma)$-coherent, i.e. there $z\in SCH(F,\Omega,\Gamma)$ minimizing $I_1$ and such that $\di z = \lambda$, $[z\cdot\nu] = \mu$ and $\mu = -\lambda$. This in turn implies that $|\lambda| = |\mu | \le 1$. This happens when $\gamma u = v$ on $\Gamma$;
\\[1ex]
or
\\[1ex]
(2) Facet $(F,\chi)$ is calibrable but not coherent, i.e. there is $z\in SCH(F,\Omega,\Gamma)$ minimizing $I_1$ and such that $\di z = \lambda$, $[z\cdot\nu] = \mu$ and $ |\mu | =1$, but $|\lambda| \ne 1$.
%\todo{I am not sure why we could know $ |\lambda| > 1 $ here, as in the previous manuscript.} 
We shall see that this cannot happen when $\gamma u = v$ on $\Gamma$.

Our task is to construct $z\in SCH(F,\Omega,\Gamma)$ for a given geometric configuration. We will prove that $z$ minimizes $I_1$ by invoking Proposition \ref{Pcon} or Lemma \ref{PCal}.

We will first look for configurations corresponding to case (1), i.e. the $(1,\Gamma)$-calibrability, the $t$-dependence is suppressed.
Due to Proposition \ref{s5.2-pro5.3} the Cahn--Hoffman vector field $z$ has the form $z(x) = w(|x|) \bne_r$, where $\bne_r = x/|x|$, for $x\neq 0$. 

The monotonicity assumption on $u_0$ gives, $\nabla u_0 \neq 0$ for $|x|\in(a,\rho),$ hence 
$|z| = \left| \nabla u / |\nabla u| \right| =1$. As a result,
$z\cdot \bne_r = w(r) = \chi$ for $|x|\in(a,\rho).$ In addition, continuity of $z$ yields,
$$
(z\cdot \nu_F)(\rho) = -\chi \quad\hbox{and}\quad [z\cdot\nu](R) = \mu,
$$
where $\nu_F$ is the outer normal to the facet $(F,\chi)$ of the annulus $A(\rho,R) := B(0, R) \setminus \bar{B}(0, \rho) $. Indeed, the value of $z$ is well-defined on $A(a,\rho)$, so the first equality holds. On the other hand, since $\gamma u -v =0$ on $\Gamma$ and $-[z\cdot\nu](R) \in \Sgn(\gamma u -v) = [-1,1]$, we conclude that $[z\cdot\nu](R) = \mu$,
where $\mu$ has to be determined. Summing up these conditions yields,
\begin{equation}\label{br_w}
w(\rho) = \chi, \qquad w(R) =\mu. 
\end{equation}
We seek $w$ such that the divergence of $z$ is constant on the facet and $|w(r)|\le 1$, for all $r\in(\rho,R)$. Since  $\di z = w' + w/r$, then we want that $w$ satisfy
\begin{equation}\label{r8.2ha}
 \di z = \frac{(r w)'}{r} = \lambda,
\end{equation}
as well as boundary conditions (\ref{br_w}), because due to Proposition \ref{s5-pr-cannon}, we need a minimal section of the subdifferential. Since we are considering  case (1), 
%\todo{I think Proposition \ref{s5-pr-cannon} is not applicable, because we do not have $ |[z \cdot \nu]| < 1 $, yet.}
%The same proposition requires that
\begin{equation}\label{r4}
 \di z = - [z\cdot\nu]\quad \hbox{for }r=R.
\end{equation}
Combining these conditions gives us 
$ \lambda = - \mu.$

After simple calculations, we reach, %$\lambda = 2\frac{R\KSh{\mu} + \rho}{R^2-\rho^2}$. Hence,
\begin{equation}\label{r6}
 %,\quad\hbox{and} \quad 
 \mu= \frac{2\rho\chi}{R^2-\rho^2 + 2R}.
\end{equation}
We can also see that $w$ has the form
$$
w(r) = \frac{\rho\chi (r^2 - R^2 -2R)}{(\rho^2 - R^2 - 2R)r},
$$
and $ |z(r)| = |w(r)| \le 1$ for $r\in [\rho,R]$. The calculations above show that $z\in SCH(F,\Omega, \Gamma)$, see (\ref{r8.2ha}) and (\ref{r4}). Moreover, we invoke Proposition \ref{Pcon} to deduce that $z$ is a minimizer of $I_\tau$, with $\tau =1$. Thus, we conclude that $(F,\chi)$ is calibrable and $(1,\Gamma)$-coherent.

The bulk at $r=\rho$ moves with velocity 
\begin{equation}\label{rnau}
 u_t = \di (\chi\bne_r) = \frac {\chi}\rho .
\end{equation}
At the same time the sign of facet velocity is $- \chi$, so we conclude that they are going in the opposite directions. As a result,  the facet expands.

%We notice that there must appear a facet containing the origin, but it does not influence our analysis until facets collide, but this happens after positive time $t_{cr}$.

We also have to follow the boundary behavior of the solution. Since
\begin{equation}\label{rnav}
 v_t = - [z\cdot\nu] = -\mu = \lambda,
\end{equation}
we see that the velocities of the facet and the boundary value are equal. This implies that $\gamma u = v$ on $\Gamma$ for $t\in[0,t_{cr}]$. 

The missing piece of information is $\rho(t)$, the inner radius of facet $( \bar B(0,R)\setminus B(0,\rho(t)), \chi)$. Since the solution $U=(u,v)$ cannot have jumps, we deduce that 
$$
v(t) = u(\rho(t),t).
$$
Taking into account the equations for $u(\rho,t)$, (\ref{rnau}), and $v(t)$, (\ref{rnav}), as well as (\ref{r6}), we deduce that
\begin{equation}\label{u=v}
v_0  - \int_0^t \frac{2\chi \rho(s)}{R^2- \rho^2(s) + 2R} \,ds =
u_0(\rho) + \frac{\chi t}{\rho}.
\end{equation}

\bigskip
Now, we consider the case $\rho=R$, i.e. a possibility of a facet formation at the boundary with the dynamic condition.
Because for a given $t\ge 0$ the selection $z$ is continuous, we notice that (while identifying $v_t$ with $\frac{d^+}{dt}v$)
$$
v_t %\KSh{~\left( = \frac{d^+}{dt}v \right)}
= -[z\cdot \nu] =-\chi \bne_r \cdot \nu = - \chi\qquad\hbox{for } t=0,
$$
where $\nu$ is the outer normal to $\d\Omega$.
If we compare $v_t$ above with $u_t$ from eq. (\ref{rnau}), then we see that
the boundary value and the bulk move in the opposite directions. We notice that $v_t =  -[z \cdot \nu] \in \Sgn(\gamma u -v) $ required by (b2) of Theorem \ref{Prop.rep_subd} (B) is possible if and only if $\gamma u =v$, thus a facet must be formed to preserving the continuity of solutions at $\Gamma$.
%\todo{I think the continuity of solution is true, as a result, but the continuous regularity needs not be adequate reason because it is not derived so immediately. }
At $t>0$, we are back in the situation we have already studied.

We might say that we
have a collision of the bulk and the boundary, which leads to the creation of a calibrable and $(1,\Gamma)$-coherent facet at $\rho =R$. 
%\sout{We notice that the absolute value of its velocity is always less or equal to 1, because of $\di z = - [z\cdot\nu]$ and $z \in \Sgn (\gamma u -v)$ holds for minimizers.}\todo{this sentence is not necessary, PR}

We %dealt with case (1)
analyzed all possible
inner radii $\rho$ of the facet $F= \overline{A(\rho, R)}$, noticing that only case (1) occurred as long as $\gamma u =v$. Thus by the uniqueness of solution,  we conclude that case (2) never happens.

The analysis of the ball is complete. We may collect our observations in a single statement.
\begin{thm}\label{Thm8.1}
 Let us assume that $\Omega$ is the ball $B(0,R)$ and the initial condition $u_0$ is radially symmetric, i.e. $u_0(x) = u_0(|x|)$, the function $r\mapsto u_0(r)$ is monotone, belongs to
 $C^2([0,R])$ and $u'_0(R)\neq 0$. 
 %Lipschitz continuous. 
 Moreover, $\Gamma = \d\Omega$ and $v_0$ is radially symmetric, i.e., $v_0$ is a constant and $v_0 = \gamma u_0$ and $\chi$ is defined in (\ref{def-chi}). Then,\\
(1) If a facet $(F,\chi)$, where $F=\overline{A(\rho,R)}$,  is present, then it is calibrable and 
$(1,\Gamma)$-coherent. Moreover, it evolves with a velocity given by
\begin{equation}\label{r7}
u_t = \di z \equiv \lambda := \frac{2\rho}{R^2-\rho^2 + 2R} \chi, 
\end{equation}
where $\chi=1$ for $u_0$ decreasing and $\chi=-1$ for $u_0$ increasing. Moreover, $|u_t|<1$, as long as $\rho<R$ and $v(t) = \gamma u(t)$ for all $t\ge 0$.
    \\[2ex]
(2) If $\rho=R$,  then a facet touching the boundary is created, which  moves according to (1) for $t>0$.
%\\[2ex]
%the bulk moves with velocity given by (\ref{r7}), i.e. $u_t=\chi = v_t$ and $v(t) = \gamma u(t)$. \\ %qed
%\todo{\KSh{I agree with this correction. KS}}   \ \\[2ex]
(3) The inner radius of the facet, $\rho(t)$, satisfies the following ODE,
\begin{equation}\label{rhoprime}
 \frac{d\rho}{dt}\left( u_0'(\rho) - \frac{\chi t}{\rho^2}\right) = -\frac \chi \rho - \frac{2\rho\chi}{R^2- \rho^2 + 2R},
\qquad \rho(0) = \rho_0.
\end{equation}
\end{thm}
\begin{proof}
Actually, we constructed explicitly $u(r,t)$, $v(t)\equiv u(R,t)$ and $z\in SCH(F, \Omega, \Gamma)$, as a section of the subdifferential $\partial E$. The computations we presented above indeed show that,
 $$
 \frac{d^+ U}{dt} = \mathcal{A}^o (U), \quad U(0) = (u_0, v_0).
 $$
 The position of the facet follows from the continuity of solutions. The differentiation of (\ref{u=v}) with respect to time yields part (3). Here it is important that $u'_0(R)\neq 0$, otherwise the derivative of $\rho$ at $t=0$ may be infinite. We note that formulas (\ref{u=v}) and (\ref{rhoprime}) are meaningful also in the case $\rho =R$.
\end{proof}

\subsection{An annulus case}\label{ssann}
In order to set the notation, we recall that we write $A(r_{in},R_{out})=B(0,R_{out})\setminus\bar B(0,r_{in})$ for an open annulus with inner radius $ r_{in}$ and outer radius $R_{out}$. We set 
$\Omega = A(r_0,R)$.
We notice that the case of a facet touching $\d B(0,R)$ is not different from the case of ball and it has been solved in the previous subsection. As previously, for the sake of definiteness, we assume that $u_0$ Lipschitz continuous, radially symmetric and monotone as a function of $|x|$.

\subsubsection{Facet evolution}

We %turn out attention to $\d B(0,r)$ and we 
assume that $\Gamma = \d B(0,r_0)$. Moreover, we also assume the existence of  the inner facet is $(F,\chi)$, where $F= \bar B(0,\rho)\setminus B(0,r_0)$, $\rho > r_0$ and $\chi$, defined in (\ref{def-chi}), indicates monotonicity of
$u_0$ in case the initial condition depends only on the distance from the origin. With the help of Proposition \ref{s5.2-pro5.3}, we deduce that the Cahn--Hoffman vector has the form of $z(x)= w(|x|)\bne_r$. However, in order to proceed, we will assume that initially $v_0$ on $\Gamma$ equals to the trace of $u_0$ on $\Gamma$. Contrary to the case of the ball, determining the  behavior of $\rho(t)$ is more difficult and we will not do this since this is not the main point here.

Now, we begin our analysis of the subdifferential. We know that this is reduced to the study of minimizers of $I_\tau$, here $\tau =1$.  In principle, as for the ball in Subsection \ref{ssball}, we have the following cases singled out in Proposition \ref{s5-pr-cannon}:\\
(1) Facet $(F,\chi)$ is calibrable and coherent, provided that $\gamma u = v$ on $\Gamma$. In other words,
there is  $z\in SCH(F,\Omega,\Gamma)$ minimizing $I_\tau$, such that $\di z = \lambda$ on $F$, $[z\cdot\nu] = \mu$ on $\Gamma$ $\mu = -\lambda$ and  $|\lambda| = |\mu | \le 1$.\\
(2) Facet $(F,\chi)$ is calibrable but not coherent, but then  $\gamma u \neq v$ on $\Gamma$.  That is, 
$z$, any minimizer of $I_\tau$ in $SCH(F,\Omega,\Gamma)$  is such that $\di z = \lambda$ on $F$, $[z\cdot\nu] = \mu$ on $\Gamma$, but $|\lambda| \ne 1$ %\todo{I am not sure why we could know $ |\lambda| > 1 $ here, as in the previous manuscript.} 
and $ |\mu | =1$. 

We will first look for configurations corresponding to case (1).
Arguing as in the previous subsection, we see that $z \in CH(F, \Omega, \Gamma)$, of the form $z = w \bne_r$ satisfies the following boundary conditions,
\begin{equation}\label{a1}
\mu = [z\cdot\nu](r_0) = - \bne_r \cdot w(r_0)\bne_r,\qquad
    (z\cdot\nu_F)(\rho) = \bne_r \cdot w(\rho)\bne_r = \chi.
\end{equation}
Moreover, we have
$$
v_t= - [z\cdot\nu](r_0) = -\mu, \qquad \di z = \lambda,   \qquad \lambda=-\mu.
$$
This yields the following boundary value problem for an ODE,
\begin{equation}\label{a2}
 \frac{(r w)'}{r}=\lambda, \quad w(r_0) = \lambda ,\quad w(\rho)=\chi.
\end{equation}
Solving this ODE yields
$$
w(r) = \frac{\lambda r}{2} + \frac cr, 
$$
where 
\begin{equation}\label{a3}
 \lambda = 2 \frac{\rho\chi}{\rho^2 - r_0^2 + 2 r_0},\qquad c = \frac \lambda 2 r_0(2- r_0).
\end{equation}
Finally,
\begin{equation}\label{a3.5}
 w(r) = \frac{\rho\chi(r^2 - r_0^2 + 2 r_0)}{r(\rho^2 - r_0^2 + 2 r_0)},\qquad 
c = \frac {\rho \chi r_0 (2-r_0)}{\rho^2 - r_0^2 + 2 r_0}.
\end{equation}
%\\[2ex]
%\todo{It would be better to mention about the governing equation of $ \rho $. In fact, the last Propositions \ref{p8.4} and \ref{p8.5} rely on the continuity of $ (u_t, v_t) $, eventually.}

%\KSh{
    \begin{remark}\label{Rem.KSh01}
        In case (1), we can derive the following ODE for $ \rho = \rho(t) $:
        \begin{equation*}
            \frac{d \rho}{dt} \left( u_0'(\rho) -\frac{\chi t}{\rho^2} \right) = -\frac{\chi}{\rho} -\frac{2 \rho \chi}{\rho^2 -r_0^2 +2 r_0},
        \end{equation*}
        just as in Theorem \ref{Thm8.1} (3). Hence, we can say that  $ (u_t, v_t) = (\lambda, -\mu) = (\lambda, \lambda) \in C(F \times [0, t]) \times C(\Gamma_F \times [0, t]) $, for any time $ t $ before the facet collision.
    \end{remark}

We have to check when $z\in SCH(F,\Omega,\Gamma)$, i.e.,
$$
|\lambda |\le 1\qquad\hbox{and} \qquad|w(r)| \le 1\hbox{ for all }r\in (r_0, \rho).
$$
In particular, we identify the case when the facet and the boundary layer move with the same velocity.
\begin{prop}\label{pr-a}
Let us suppose that $\Omega = A(r_0,R)$, $\Gamma = \d B(0,r_0)$, $u_0(x) = \tilde u_0(r)$, $\rho>r_0$, $u_0$ is of $C^2$-class  
%\todo{I think it would be safe to replace ``Lipschitz continuous'' by ``in $ C^2 $-class'' to see the existence-uniqueness of the governing ODE of $ \rho $, easily. } %Lipschitz continuous 
    and $\tilde u_0$ is monotone. We also assume that $(F,\chi)$ is a facet, where $F = \overline{A(r_0,\rho)}$ and $\chi$ is given by (\ref{def-chi}). Moreover, 
$v_0 = \tilde u_0(r_0)$ and $ z(x,t) = w(|x|,t)\frac x{|x|}$, where $w$ is given by (\ref{a3.5}), $\lambda$ is defined by (\ref{a3}).

If $r_0>2$,
%all the assumptions specified above are valid, 
then $|\lambda|\le 1$
and $|w(r)| \le 1$ for all $r\in (r_0, \rho)$,  as a result  $z \in SCH(F,\Omega, \Gamma)$. Moreover, $z$ is the minimal selection $\mathcal{A}^o(U)$, i.e. facet $(F,\chi)$ is calibrable and $(1,\Gamma)$-coherent.
\end{prop}
\begin{proof}
The formulas (\ref{a3}) and (\ref{a3.5}) for $\lambda$ and $w$ were derived on the premise that $\di z = - [z\cdot\nu]$ and  $\di z$ is a constant.

It is easy to see that  $|z(x)|=|w(|x|)| \le 1$ if and only if $\rho + r_0 \ge 2$.
Moreover, if $\lambda$ is given by (\ref{a3}), then 
$ |\lambda| \leq 1 $ is equivalent to $ 0 \leq (\rho -r_0)(\rho +r_0 -2) $, and the equality holds if and only if $ \rho +r_0  = 2 $.    
%$| \lambda| <  1$ is equivalent to $0 < (\rho - r_0)( \rho + r_0 -2)$. 
    The last inequality is true if and only if  $\rho+ r_0>2$. Thus, $z(x,t) = w(|x|,t)\bne_r$ is in $SCH(F,\Omega,\Gamma)$, where $F =\bar B(0,\rho)\setminus B(0,r_0)$. We invoke Proposition \ref{Pcon} to deduce that  $z$ minimizes $I_\tau$ with $\tau=1$. Hence, the facet $(F,\chi)$ is calibrable and coherent. Moreover, we see that $v_t = \gamma u_t$ on $\Gamma$.

%\textcolor{blue}{We notice that if $\rho = r_0$, then $\lambda =1$, hence at the instant of a facet creation it moves with the velocity of the boundary layer.}
\end{proof}
This proposition states that the facet and the boundary layer move at the same velocity.

\bigskip
The borderline case $r_0=2$ behaves like the previous one with some changes. We see from (\ref{a3}) that
$\lambda = 2\frac \chi \rho$, thus $|\lambda| < 1$, because $\rho>r_0=2$. Moreover, formula (\ref{a3}) implies that $c=0$, hence, 
$$
w(r) = \frac {\chi r}\rho,
$$
so $|w(r)| \le 1$. We collect these observations below.
\begin{prop}
Let us suppose that the hypotheses of Proposition \ref{pr-a} hold, but
$r_0=2$. %and all the assumptions specified above are valid, t
Then, $\lambda = 2\chi/\rho$, facet $(F,\chi)$, where $F$ is defined in the proposition above and  is calibrable and coherent.  Moreover, $\chi = v_t = \gamma u_t$ on $\Gamma$.
\end{prop}
\begin{proof} Clearly, $z(x,t)= w(|x|,t)\frac x{|x|}$, where $w$ is given above, belongs to $SCH(F,\Omega,\Gamma)$. Moreover, $|\lambda| = 2/\rho<1$, because $\rho>r_0=2$. Hence, Proposition \ref{Pcon} implies that $z$ minimizes $I_\tau$, with $\tau =1$. As a result, we obtain the canonical section of $\d E$. Thus, facet $(F,\chi)$, where $F = \bar B(0,\rho) \setminus B(0,r_0)$, is calibrable and coherent. We also see that the velocities $v_t$ and $u_t$ on the facet are equal. 
\end{proof}
%\vspace{-16ex}

%\todo{I think in we should explain all the cases of $r_0$ - PR.}
%\todo{\KSh{I agree with you in this time. Now I could see the point I misunderstood. KS}}
%\vspace{16ex}

%this subsection, it is sufficient to prove only the following Proposition \ref{Prop.final}. In fact, the branching of cases $ r_0 > 2 $, $ r_0 =2  $ and $ r_0 < 2 $ looks not so essential for me.}
More computations are required when $r_0<2$, they are presented in the course of proof of the proposition below.
\begin{prop}\label{Prop.final}
Let us suppose that the hypotheses of Proposition \ref{pr-a} hold, but
$r_0<2$.  Then, $\lambda$ is given by (\ref{a5}) below and:\\
1) $|\lambda|>1$ is equivalent to $\rho + r_0<2$. Then, 
facet $(F,\chi)$ is calibrable, but not $(1,\Gamma)$-coherent and
$|u_t|> 
|v_t|$. Since  $v_t$ and $u_t$ have the same sign, as a result for $t>0$ the boundary layer detaches.\\
2) $|\lambda|=1$ is equivalent to $\rho + r_0=2$, and  facet $(F,\chi)$ is calibrable and $(1,\Gamma)$-coherent
i.e. the boundary layer moves at the velocity of the bulk.\\
3) $|\lambda|<1$ is equivalent to $\rho + r_0>2$. Then, facet $(F,\chi)$ is calibrable and $(1,\Gamma)$-coherent.
%$|u_t| = |v_t|$. Since  $v_t$ and $u_t$ have the same sign, as a result for $t>0$ a facet is formed. At $t>0$ the new facet is  calibrable. 
%\todo{\KSh{I have the same comment with (18) of Prof. Giga's comment: Proposition 8.4 3) and the last part of its proof is inconsistent. KS}}
\end{prop}
\begin{proof}
We noticed in the course of proof of Proposition \ref{pr-a} that formulas (\ref{a3}) for $\lambda$ yielded 
$|\lambda|\le1 $ if and only if $r_0+\rho \ge 2$. In other words, if $r_0+\rho < 2$, then (\ref{a3}) and (\ref{a3.5}) are no longer correct, because they violate the condition $|w(r_0)|\le 1$.
%$\di z$ and  $[z\cdot\nu]$ need not be related and we even expect $|\lambda| >1$. In the present case
Thus, we consider equation (\ref{a2}) for 
$w$, but with the boundary conditions specified below,
\begin{equation}\label{a4}
 \frac{(r w)'}{r}=\lambda, \quad w(r_0) = \chi,\quad w(\rho)=\chi.
\end{equation}
Obviously,  we need to determine $\lambda$ to be able to solve (\ref{a4}).

In order to justify the condition $w(r_0) = \chi$, we recall that if we had $|w(r_0)| = | [z\cdot\nu]| <1$, then we would deduce from Proposition \ref{s5-pr-cannon} that $\di z = - [z\cdot\nu]$, but this occurs when $r_0+\rho\ge 2$.

For the purpose of determining $\lambda$ we use Lemma \ref{PCali}. 
This leads to the following condition on $\lambda$, assuming $ u_t$ is the velocity of facet $F$,
\begin{equation}\label{a4.5}
    \int_F u_t = \int_F \di z = \int_{\d F} [z\cdot \nu_F]\, d\cH^1 =
\chi \int_{\d B(0,\rho)}\,d\cH^1 - \chi \int_{\d B(0,r_0)}\,d\cH^1.
\end{equation}
Since we look for $z$ satisfying $\di z = \lambda$, then we obtain
$$
\lambda \pi (\rho^2 - r^2_0) = 2\pi \chi( \rho - r_0).
$$
As a result,
\begin{equation}\label{a5}
  \lambda = \frac{2\chi}{\rho + r_0}
\end{equation}
and its absolute value exceeds 1 if and only if $\rho + r_0<2$. 

We easily see that the solution to (\ref{a4}) is
\begin{equation}\label{a4.75}
 w(r) = \frac{\chi(r^2+ \rho r_0)}{r(\rho+r_0)};
\end{equation}
We must make sure that $|w(r)|\le 1$.
We notice that this condition is equivalent to 
$$
    (\rho - r)(r_0- r) \le 0
$$
which is always true for $r\in [r_0,\rho]$ with equalities at $r= r_0$  or $r=\rho$. 

Formula (\ref{a4}) also shows that $\mu = -\chi$. Thus, we may apply Lemma \ref{PCal} to deduce that $z\in SCH(F, \Omega, \Gamma)$, we constructed, minimizes $I_1$. Thus, facet $(F,\chi)$ is calibrable, but not coherent because $\lambda + \mu \ne 0$.

At the same time due to the boundary conditions (\ref{a4}), we notice that $v_t = \chi$. In other words the boundary layer is slower than the bulk and it detaches, due to the lack of coherency.

If $\rho + r_0=2$, then the Cahn--Hoffman vector field, we constructed above, minimizes $I_1$ due to Lemma \ref{PCal}. Moreover, $\lambda + \mu =0$, hence facet $(F,\chi)$ is calibrable and coherent. Finally,
the velocities of the bulk at $r=r_0$ and the boundary layer are equal, so the solution is continuous.

The condition $\rho + r_0 >2$ is equivalent to $|\lambda| =| u_t|<1$. 
In this case we proceed as in the proof of Proposition \ref{pr-a} and solve equation (\ref{a2}). Then $\lambda$ is given by (\ref{a3}) while (\ref{a3.5}) gives $w$. In the course of proof  of Proposition \ref{pr-a}, we learned that $|w(r)|\le 1$ for $r\in(r_0,\rho)$ if and only if $r_0+\rho>2$. Now, Proposition \ref{Pcon} yields minimality of $z(x) = w(|x|) \bne_r$, hence facet $(F,\chi)$ is calibrable. Moreover, since $\lambda + \mu =0$, the facet is $(1,\Gamma)$-coherent.
%$\mu = -\chi$ and the proof of Lemma \ref{PCal} shows that we may take a vector field $\zeta$ so that $z+ \zeta \in CH(F, \Omega, \Gamma)$ and
%\begin{equation}\label{a5.5}
%\int_{\Gamma_F} [\zeta\cdot \nu]\mathrm{d} \cH^1  \Sgn\mu <0.
%\end{equation}
%Since $\Sgn( \lambda + \mu) = \Sgn\mu$, then we deduce from (\ref{l6.4eq1}) that $z = w\bne_r$, where $w$ is given by (\ref{a4.75}) cannot minimize $I_1$. In other words, we must consider again $w$ given by (\ref{a3.5}) which satisfies (\ref{a2}) with $\lambda$ provided by (\ref{a3}). The proof of Proposition \ref{pr-a} is applicable here, as a result we conclude that if $\rho + r_0 >2$,  then facet $(F,\chi)$ is calibrable and $(1,\Gamma)$-coherent.
%Since $v_t = \chi$ and $v_t$, $u_t$ have the same sign, then we notice that the boundary layer moves faster than the bulk. That is, for $t>0$ the facet expands.
\end{proof}

\begin{remark}
 We notice that in part 1) the curvature of $\Gamma$ may be in the interval $(-\infty, -1)$. In case 3), we have calibrable and coherent facet $\bar A(r_0,\rho )$ whose curvature is in the interval $(-1,0)$, which is in accordance with Conjecture \ref{e-con}.
\end{remark}

\subsubsection{Boundary phenomena when $r_0 =\rho$}
We want to analyze the situation occurring when the initial data contains no facet and $\Omega = A(r_0,R)$ while $\Gamma = \d B(0,r_0)$ and $u_0|_{\Gamma} = v_0$. Exactly, as in the previous subsection, we do not address the question of the position of the inner radius of the facet.

We keep in mind that on the one hand
\begin{equation}\label{a6}
v_t = - [z\cdot\nu], %= W_p(\nabla u)\cdot\nu,
\end{equation}
while on the other hand $[z\cdot\nu] \in \Sgn(v - \gamma u).$ Meanwhile, since $z$ is continuous and $z = \chi \bne_r$ at $ t = 0 $, 
  we deduce that (we write $\frac{d^+}{dt}v =v_t $),
$$
%\left( \frac{d^+}{dt}v = \right) 
v_t  = \chi \quad \mbox{at $ t = 0 $.}
$$
We may calculate the bulk velocity, (here $\frac{d^+}{dt}u = u_t$),
\begin{equation}\label{a7}
%\KSh{\left( = \frac{d^+}{dt}u \right)=~} 
u_t = \di z =\frac{\chi}{r} \quad \mbox{at $ t = 0 $.}
\end{equation}
Thus, we are ready to state the final result:

\begin{prop}\label{p8.4}
 Let us suppose that $\Omega = A(r_0,R)$, $\Gamma = \d B(0,r_0)$ and $u_0(x) = u_0(|x|)$, where function $r\mapsto u_0(r)$ is strictly monotone on $[r_0, R)$, $R>r_0$. We assume that no facet is contained in the data, moreover, $u_0(r_0) = v_0$. Then,\\
 1) If $r_0> 1$, then $ 0 < \frac{2R}{R^2 -r_0^2 +2r_0} < |\gamma u_t| = |v_t| \leq 1 $ on $ \Gamma $, for $t$ close to zero, so that a facet forms;\\
 2) If $r_0 = 1$, then $ \gamma u_t = v_t = \chi $ on $ \Gamma $, for $t$ close to zero, and no facet is created;\\
 3) If $r_0 <1$, then $| \gamma u_t| > |v_t| = 1 $ on $ \Gamma $, for $t$ close to zero, so that the boundary layer detaches.
\end{prop}
\begin{proof}
We have already calculated  the bulk velocity in (\ref{a7}).
We want to examine the speed of the bulk at $r=r_0$,
$$
  |\gamma u_t | = | \di z | = \frac 1 { r_0} \quad \mbox{at $ t = 0 $.}
$$
We also noticed that
$$
v_t = \chi \quad \mbox{at $ t = 0 $.}
$$
Thus, we see that the boundary layer moves faster than the bulk if and only if $r_0 >1$, at the same time $\gamma u_t v_t\ge 0$. This means that a facet is formed and for $t>0$. Furthermore, the results obtained earlier are applicable. Therefore, 
with \eqref{a3} and  $ 1 < r_0 \leq \rho < R $ in mind, we can see that
    \begin{equation*}
        \begin{array}{c}
            \ds 0 < \frac{2R}{R^2 -r_0^2 +2r_0} < \frac{2 \rho}{\rho^2 -r_0^2 +2 r_0} = |\gamma u_t| = |v_t| \leq 1, 
        \end{array}
    \end{equation*}
until the facet collision occurs.

If $r_0=1$, then we see that $ \gamma u_t =v_t = \chi $. 
As a result, the boundary layer moves with the bulk, and no new facet is created.

Finally, let us consider $r_0<1$, then 
$$
|\gamma u_t | = | \di z | = \frac 1{r_0}>1 \quad \mbox{at $ t = 0 $.}
$$
On the other hand
$$
v_t = - [z\cdot\nu] = \chi \quad \mbox{at $ t = 0 $.}
$$
    In other words, the bulk moves faster, so $v - \gamma u \neq 0$ and $\Sgn(v - \gamma u)=- \chi.$ As a result, the boundary layer detaches and no facet forms, and $ |\gamma u_t| > |v_t| = 1 $, for $t>0$ close to zero.
    %cally in time, i.e. at least, until the bulk disappears.}
\end{proof}

The observations we have made above imply the energy decay.
\begin{prop}\label{p8.5}
Under the assumptions of Proposition \ref{p8.4}, the total energy decays, i.e. $t\mapsto E(u(t), v(t))$ is a decreasing function, for $t>0$ close to zero.
\end{prop}
\begin{proof}
        By Proposition \ref{p8.4}, we can find a small positive constant $ \delta_0 $ such that
        \begin{equation*}
            |v_t| \geq \frac{2R}{R^2 -r_0^2 +2 r_0} > 0 \quad \mbox{on $ [0, \delta_0] $.}
        \end{equation*}
        Additionally, from \eqref{rn2} in Theorem \ref{Thm4.4}, we can see that
        \begin{align*}
            \frac{d}{dt} E(u(t), v(t)) ~ & =  -|u_t(t)|_{L^2(\Omega)}^2 -|v_t(t)|_{L^2(\Gamma)}^2
            \\
            & \leq -2 \pi r_0 |v_t(t)|^2  < -\frac{8 \pi r_0 R^2}{(R^2 -r_0^2 +2r_0)^2}< 0 \quad \mbox{a.e. $ t \in (0, \delta_0) $.}
        \end{align*}

        \noindent
        It implies that the function $ t \mapsto E(u(t), v(t)) $ is decreasing on the time-interval $ [0, \delta_0] $.
%Indeed, 
%\begin{eqnarray*}
% &&\frac d{dt} \left(
% \int_\Omega |\nabla u|  + \int_\Gamma |u-v|\,\mathrm{d}\cH^{N-1}\right) =
%\int_\Omega \left(\frac{\nabla u}{|\nabla u|}\right)\nabla u_t + \int_\Gamma \Sgn(u-v)(\gamma u_t - v_t)\,\mathrm{d}\cH^{N-1}\\
%&=&
%-\int_\Omega \di \left(\frac{\nabla u}{|\nabla u|}\right) u_t\mathrm{d}x
%+ \int_\Gamma [\left(\frac{\nabla u}{|\nabla u|}\right)\cdot \nu \gamma u_t +\Sgn(u-v)(\gamma u_t - v_t)]\,\mathrm{d}\cH^{N-1}\\
%&=&
%-\int_\Omega u_t^2 \,\mathrm{d}x + \int_\Gamma [- v_t \gamma u_t +\Sgn(u-v)(\gamma u_t - v_t)]\,\mathrm{d}\cH^{N-1}\\
%&=:& R.
%\end{eqnarray*}
%Since by (\ref{a5}) $v_t=\chi$ we conclude that
%$$
%R= -\int_\Omega u_t^2 \,\mathrm{d}x + \int_\Gamma [- \gamma u_t +(\gamma u_t - v_t)]\,\mathrm{d}\cH^{N-1}<0.
%$$
%Hence, $E(u(t),v(t))$ strongly decreases.
\end{proof}

%\section{Appendix on $BV$ functions}

\section*{Acknowledgment}
The work of the first author, Yoshikazu Giga, was in part supported by the Japan Society for the Promotion of Science (JSPS) through the grants Kiban (S) (No. 26220702), Kiban (A) (No. 17H01091), Kiban (B) (No. 16H03948) and Challenging Research (Pioneering) (No. 18H05323).
The work of the third  author, Piotr Rybka, was in part supported by the National Science Centre, Poland, through the grant number
2017/26/M/ST1/00700.
The fourth author, Ken Shirakawa, was supported by Grant-in-Aid Kiban (C) (No. 16K05224), JSPS.

%\bibliography{sinst10}

\begin{comment}

\end{comment}
\end{document}